\documentclass[reqno,a4paper]{amsart} 
\usepackage{amsmath, latexsym, amsthm, amsfonts,bm,amssymb} 
\usepackage{graphicx} 
\usepackage{float}
\usepackage{appendix}
\usepackage{natbib} 
\usepackage{enumerate}
\usepackage{mathtools} 
\usepackage[utf8]{inputenc}

\usepackage[colorlinks,linkcolor=blue, citecolor=blue,urlcolor=blue]{hyperref}
\usepackage[dvipsnames]{xcolor}
\usepackage{booktabs} 
\usepackage{caption} 

\bibpunct{(}{)}{;}{a}{}{,} 

\theoremstyle{plain}
\newtheorem{thm}{Theorem}
\newtheorem{prop}{Proposition}
\newtheorem{lemma}{Lemma}

\theoremstyle{remark}
\newtheorem{rem}{Remark}

\def\iid{\stackrel{\textrm{i.i.d.}}{\sim}}

\newcommand{\pp}{\mathbb{P}}
\newcommand{\ee}{\mathbb{E}}
\newcommand{\rr}{\mathbb{R}}
\newcommand{\N}{\mathbb{N}}

\newcommand*\dd{\mathop{}\!\mathrm{d}}
\newcommand{\ig}{\operatorname{IG}}

\newcommand{\geom}{\operatorname{Geom}}
\newcommand{\ind}{\mathbf{1}}
\newcommand{\pr}{\operatorname{Pr}}
\newcommand{\la}{\lambda}
\newcommand{\scr}[1]{{\mathcal #1}}
\newcommand{\eps}{\epsilon}
\allowdisplaybreaks

\usepackage{natbib} 
\bibpunct{(}{)}{;}{a}{}{,} 

\usepackage[foot]{amsaddr} 

\usepackage{tikz,xcolor,hyperref}

\definecolor{lime}{HTML}{A6CE39}
\DeclareRobustCommand{\orcidicon}{%
	\begin{tikzpicture}
	\draw[lime, fill=lime] (0,0) 
	circle [radius=0.16] 
	node[white] {{\fontfamily{qag}\selectfont \tiny ID}};
	\draw[white, fill=white] (-0.0625,0.095) 
	circle [radius=0.007];
	\end{tikzpicture}
	\hspace{-2mm}
}

\foreach \x in {A, ..., Z}{%
	\expandafter\xdef\csname orcid\x\endcsname{\noexpand\href{https://orcid.org/\csname orcidauthor\x\endcsname}{\noexpand\orcidicon}}
}

\begin{document}

\linespread{1.25}

\title[Decompounding discrete distributions]{Decompounding discrete distributions: A non-parametric Bayesian approach}

\author[S.~Gugushvili]{Shota Gugushvili$^1$\orcidA{}}
\address{$^1$Biometris\\
	Wageningen University \& Research}
\email{shota@yesdatasolutions.com}

\author[E.~Mariucci]{Ester Mariucci$^2$\orcidB{}}
\address{$^2$Institut f\"{u}r Mathematik\\
Potsdam Universit\"{a}t}
\email{mariucci@ovgu.de}

\author[F. H.~van der Meulen]{Frank van der Meulen$^3$\orcidC{}}
\address{$^3$Delft Institute of Applied Mathematics\\
	Delft University of Technology}
\email{f.h.vandermeulen@tudelft.nl}

\thanks{The research leading to the results in this paper has received funding from the European Research Council under ERC Grant Agreement 320637, from
the Deutsche Forschungsgemeinschaft (DFG, German Research Foundation) -- 314838170,
GRK 2297 MathCoRe, and from the Deutsche Forschungsgemeinschaft (DFG) through the grant CRC 1294 `Data Assimilation'}


\keywords{Compound Poisson process; Data augmentation;  Diophantine equation; Gibbs sampler; L\'evy measure; Metropolis-Hastings algorithm; Non-parametric Bayesian estimation; Posterior contraction rate}

\begin{abstract}

Suppose that a compound Poisson process is observed discretely in time and assume that its jump distribution is supported on the set of natural numbers. In this paper we propose a non-parametric Bayesian approach to estimate the intensity of the underlying Poisson process and the distribution of the jumps. We provide a MCMC scheme for obtaining samples from the posterior.  We apply our method on both simulated and real data examples, and compare its performance with the frequentist plug-in estimator proposed by Buchmann and Gr\"ubel. On a theoretical side, we study the posterior from the frequentist point of view and prove that as the sample size $n\rightarrow\infty$, it contracts around the `true', data-generating parameters at  rate $1/\sqrt{n}$, up to a $\log n$ factor. 

\end{abstract}


\maketitle

\section{Introduction}
\label{intro}

\subsection{Problem formulation}

Let $N=(N_{t} : t\geq 0)$ be a Poisson process with a constant intensity $\lambda>0$, and let $Y_i$ be a sequence of independent random variables, each with distribution $P$, that are also independent of $N$. By definition, a compound Poisson process (abbreviated CPP) $X=(X_t : t\geq 0)$ is 
\begin{equation}
\label{cpp}
X_t=\sum_{j=1}^{N_t} Y_j,
\end{equation}
where here and below the sum over an empty index set is understood to be equal to zero. In particular, $X_0=0$. CPP constitutes a classical model in, e.g., risk theory, see~\cite{embrechts97}.

Assume that the process $X$ is observed at discrete times $0<t_1<t_2<\ldots<t_n=T$, where the instants $t_i$ are not necessarily equidistant on $[0,T]$. Based on the observations $X_{t_1},X_{t_2},\ldots,X_{t_n}$, our goal is to estimate the jump size distribution $P$ and the intensity $\lambda$.  We specifically restrict our attention to the case where $P$ is a discrete distribution, $P(\mathbb{N})=1$, and we will write $p=(p_k)_{k\in\mathbb{N}}$ for the probability mass function corresponding to $P$, where $p_k=P(\{k\})$. A similar notation will be used for any other discrete law. The distribution $P$ is called the base distribution. Abusing terminology, we will at times identify it with the corresponding probability mass function $p.$ An assumption that $P$ has no atom at zero is made for identifiability: otherwise this atom gets confounded with $e^{-\lambda}$, which does not allow consistent estimation of the intensity $\lambda$. For a discussion of applications of this CPP model in risk theory, see \cite{zhang14}.

Define the increments $Z_i = X_{t_i} - X_{t_{i-1}}$, $i=1,\ldots,n$. Then $\mathcal{Z}_n = (Z_i:i = 1,\ldots,n)$ is a sequence of  independent random variables. When $\{t_i\}$ are equidistant on $[0,T]$, the random variables $Z_i$ have in fact a common distribution $Q$ satisfying $Q(\mathbb{N}_0)=1$. As $\mathcal{Z}_n$ carries as much information as $(X_{t_i}:i=1,\dots,n)$ does, we can base our estimation procedure directly on the increments $\mathcal{Z}_n$. Since summing up the jumps $Y_j$'s amounts to compounding their distributions, the inverse problem of recovering $P$ and $\lambda$ from $Z_i$ can be referred to as decompounding; see \cite{buchmann03}.

There are two natural ways to parametrise the CPP model: either in terms of the pair $(\lambda,p)$, or in terms of the L\'evy measure $\nu=(\nu_k)_{k\in\mathbb{N}}$ of the process $X$, see \cite{sato13}. A relationship between the two is $\lambda=\sum_{k=1}^{\infty}\nu_k$ and $p=\nu/\lambda$. Inferential conclusions in one parametrisation can be easily translated into inferential conclusions into another parametrisation. However, for our specific statistical approach the L\'evy measure parametrisation turns out to be more advantageous from the computational point of view.

\subsection{Approach and results}

In this paper, we take a non-parametric Bayesian approach to estimation of the L\'evy measure $\nu$ of $X.$ See \cite{ghosal17} and \cite{mueller15} for modern expositions of Bayesian non-parametrics. A case for non-parametric Bayesian methods has already been made elsewhere in the literature, and will not be repeated here. On the practical side, we implement our procedure via the Gibbs sampler and data augmentation, and show that it performs well under various simulation setups. On the theoretical side, we establish its consistency and derive the corresponding posterior contraction rate, which can be thought of as an analogue of a convergence rate of a frequentist estimator (see \cite{ghosal17}). The posterior contraction rate, up to a practically insignificant $\log n$ factor, turns out to be $1/\sqrt{n}$, which is an optimal rate for non-parametric estimation of cumulative distribution functions. Our contribution thus nicely bridges practical and theoretical aspects of Bayesian non-parametrics.

\subsection{Related literature}

To provide a better motivation for our model and approach, in this subsection we briefly survey the existing literature. A Bayesian approach to non-parametric inference for L\'evy processes is a very recent and emerging topic, with references limited at the moment to \cite{belomestny18}, \cite{gugu15}, \cite{gugu18} and \cite{nickl17}. These deal exclusively with the case when the L\'evy measure is absolutely continuous with respect to the Lebesgue density. At least from the computational point of view, these works are of no help in our present context.

Related frequentist papers for CPP models with discrete base distributions are \cite{buchmann03} and \cite{buchmann04}, which, after earlier contributions dating from the previous century, in fact revived interest in non-parametric techniques for L\'evy processes. To estimate the base distribution $p$, \cite{buchmann03} employ a frequentist plug-in approach relying on the  Panjer recursion (i.e., an empirical cumulative distribution estimate of $q$ is plugged into the Panjer recursion equations to yield an estimate of $p$; see below on the Panjer recursion). The drawback is that the parameter estimates are not guaranteed to be non-negative. \cite{buchmann04} fix this problem by truncation and renormalisation. This works, but looks artificial. As noted in \cite{buchmann04}, in practice the latter approach breaks down if no zero values are observed among $Z_i$'s. \cite{buchmann04} establish weak convergence of their modified estimator, but on the downside its asymptotic distribution is unwieldy to give confidence statements on $p$. Most importantly, the plug-in approaches in \cite{buchmann03} and \cite{buchmann04} do not allow obvious generalisations to non-equidistant observation times $\{t_i\}.$ In \cite{lindo18}, another frequentist estimator of the jump measure is introduced, that is obtained via the steepest descent technique as a solution to an optimisation problem over the cone of positive measures. The emphasis in \cite{lindo18}  is on numerical aspects; again, no obvious generalisation to the case of non-equidistant $\{t_i\}$ is available.

Finally, some important, predominantly theoretical references on inference for L\'evy processes are \cite{comte11}, \cite{duval11}, \cite{kappus14}, \cite{neumann09}, \cite{nickl12}, \cite{vanes07} and \cite{trabs15}. We refer to \cite{belomestny15}, \cite{coca18}, \cite{coca18_cpp} and \cite{duval17} for more extensive literature surveys.

\subsection{Outline}
The rest of the paper is organised as follows: in Section \ref{sec:algorithm} we introduce our approach and describe an algorithm for drawing from the posterior distribution. In Sections \ref{sec:examples} and \ref{sec:realdata} we study its performance on synthetic and real examples. Section \ref{sec:asymptotics} is devoted to the examination of asymptotic frequentist properties of our procedure. An outlook on our results is given in Section \ref{sec:outlook}. Finally, in Appendix \ref{app:lemmas} technical lemmas used in the proofs of Section \ref{sec:asymptotics} are collected, whereas Appendix \ref{app:bg:boot} contains some additional simulation results.

\subsection{Notation}
For two sequences $\{a_n\}$ and $\{b_n\}$ of positive real numbers, the notation $a_n\lesssim b_n$  (or $b_n\gtrsim a_n$) means that there exists a constant $C>0$ that is independent of $n$ and such that $a_n\leq C b_n.$ We write $a_n\asymp b_n$ if both $a_n\lesssim b_n$ and $a_n\gtrsim b_n$ hold.
We denote a prior (possibly depending on the sample size $n$) by $\Pi_n$. The corresponding posterior measure is denoted by  $\Pi_n(\cdot\mid \mathcal{Z}_n).$ The Gamma distribution with shape parameter $a$ and rate parameter $b$ ($a,b >0$) is denoted by $\operatorname{Gamma}(a,b)$. Its density is given by
$x \mapsto \frac{b^a}{\Gamma(a)} x^{a-1}e^{-bx}, \quad x>0$,
where $\Gamma$ is the Gamma function. The inverse Gamma distribution with shape parameter $a$ and scale parameter $b$ is denoted by $\ig(a,b)$. Its density is 
$x \mapsto \frac{b^a}{\Gamma(a)} x^{-a-1}e^{-b / x}, \quad x>0$. We use the notation $\operatorname{Exp}(a)$ for an exponential distribution with mean $1/a$.
Finally, given a metric $d$ on a set $\mathcal Q$ and $\epsilon > 0$, the covering number $N(\epsilon,\mathcal Q, d)$ is defined as the minimal number of balls of radius $\epsilon$ needed to cover $\mathcal Q$.
\section{Algorithm for drawing from the posterior}
\label{sec:algorithm}

A Bayesian statistical approach relies on the combination of the likelihood and the prior on the parameter of interest through Bayes' formula. We start with specifying the prior. As far as the likelihood is concerned, although explicit, it is intractable from the computational point of view for non-parametric inference in CPP models. We will circumvent the latter problem by means of data augmentation, as detailed below.

\subsection{Prior}
\label{subsec:prior}
We define a prior $\Pi$ on $\nu$ through a hierarchical specification 
\begin{align*}
\{\nu_k\}_{k=1}^\infty  \mid a, m, \beta_k &  \iid \operatorname{Gamma}(a,1/\beta_k)	\cdot \mathbf{1}_{\{1\le k\le m\}}, \\ 
	\beta_1,\ldots, \beta_m \mid \gamma & \iid\ig(c,\gamma),\\
	\gamma &\sim \operatorname{Exp}(1).
\end{align*}	
 Note that the (fixed) hyperparameters $m\in\mathbb{N}$, $a, c >0$ are denoted by  Latin letters. 

The hyperparameter $m$ incorporates our  a priori opinion on the  support of the L\'evy measure $\nu$, or equivalently, the base measure $p$. In applications, the support of $p$ may be unknown, which necessitates the use of a large $m$, e.g.\ $m=\max_{i=1,\ldots,n} Z_i$; this latter is the maximal value suggested by the data $\mathcal{Z}_n$ at hand. Nevertheless, we may simultaneously expect that the `true', data-generating $\nu$ charges full mass only to a proper, perhaps even a small subset of the set $\{1,\ldots,m\}$. In other words, $\nu$ may form a sparse sequence, with many components equal to zero. In fact, there are at least two plausible explanations for an occurrence of a large increment $Z_i$ in the data: either a few large jumps $Y_j$'s occurred, which points towards a large right endpoint of the support of $\nu_0$, or $Z_i$ is predominantly formed of many small jumps, which in turn indicates that the intensity $\lambda$ of the Poisson arrival process $N$ may be large. To achieve accurate estimation results, a prior should take a possible sparsity of $\nu$ into account. This is precisely the reason of our hierarchical definition of the prior $\Pi$: a small $\beta_k$ encourages a priori the shrinkage of the components $\nu_k$ of $\nu$ towards zero.


\subsection{Data augmentation}

Assume temporarily $t_i=i,$ $i=1,\ldots,n$, and write $q=(q_k)_{k\in\mathbb{N}_0}$ for $q_k=q(\{k\})$. Then $Z_i$ have the distribution
\begin{equation}
	\label{eq:comp}
	q = e^{-\lambda} \sum_{j=0}^{\infty} \frac{\lambda^j}{j!}p^{\ast j},
\end{equation}
with $*$ denoting convolution.  The compounding mapping $(\lambda,p)\mapsto q$ can be expressed explicitly via the Panjer recursion (see \cite{panjer81}):
\[
q_0 = e^{-\lambda}, \quad q_k = \frac{\lambda}{k}\sum_{j=1}^{k} j p_j q_{k-j}, \quad k\in\mathbb{N}.
\]
This recursion can be inverted to give the inverse mapping $q\mapsto(\lambda,p)$ via
\[
\lambda = - \log q_0, \quad p_k = - \frac{q_k}{q_0\log q_0} - \frac{1}{kq_0}\sum_{j=1}^{k-1}jp_jq_{k-j}, \quad k\in\mathbb{N}.
\]

In view of \eqref{eq:comp}, the likelihood in the CPP model is explicit. Nevertheless, an attempt to directly use \eqref{eq:comp} or the Panjer recursions in posterior computations results in a numerically intractable procedure. Equally important is the fact that a Panjer recursion based approach would not apply to non-equidistant observation times $\{t_i\}$. Therefore, instead of \eqref{eq:comp} and the Panjer recursion, we will employ  data augmentation, see \cite{tanner87}. We switch back to the case when $\{t_i\}$ are not necessarily uniformly spaced.  The details of our procedure are as follows: when the process $X$ is observed continuously over the time interval $[0,T]$, so that our observations are a full sample path $X^{(T)} = (X_t:t\in[0,T])$ of CPP, the likelihood is tractable and is proportional to
\[
e^{-T\sum_{k=1}^m\nu_k} \prod_{k=1}^m \nu_k^{\mu_k},
\]
see \cite{shreve04}, p.~498. Here
\begin{equation}
\label{eq:muk}
\mu_k = \# \{Y_j=k\},
\end{equation}
i.e.\ the total number of jumps of size $k$.
Then the  prior $\Pi$ from Subsection \ref{subsec:prior} leads to conjugate posterior computations. In fact, the full conditionals are
\begin{gather*}
\nu_k \mid \{\mu_k\},\, \{\beta_k\} \iid \operatorname{Gamma}(a+\mu_k,1/\beta_k+T), \quad k=1,\ldots,m,\\
\beta_k \mid \{\nu_k\},\, \gamma \iid \ig(a + c,\gamma + \nu_k), \quad k=1,\ldots,m,\\
\gamma \mid \{\beta_k\} \sim \operatorname{Gamma}\left(c m+1, 1+\sum_{k=1}^m \beta_k^{-1}\right).
\end{gather*}
Therefore, the Gibbs sampler for posterior inference on $\nu$ can be implemented. The Gibbs sampler cycles through the above conditionals a large number of times, generating approximate (dependent) samples from the posterior. See, e.g.,  \cite{gelfand90} and Section 24.5 in \cite{wasserman04} on the Gibbs sampler and its use in Bayesian statistics.


As we do not observe the process $X$ continuously, we will combine the above with the data augmentation device. First note that we have
\[
Z_i = \sum_{j=1}^m j \mu_{ij},
\]
where $(\mu_{ij}:i=1,\ldots,n, j=1,\ldots,m)$ are independent, and $\mu_{ij} \sim \operatorname{Poisson}(\Delta_i\nu_j)$ for $\nu_j = \lambda p_j$ and $\Delta_i = t_i - t_{i-1}$; see Corollary 11.3.4 in \cite{shreve04}. Furthermore, for $\mu_k$ as in \eqref{eq:muk} we trivially have $\mu_k=\sum_{i=1}^n \mu_{ik}$. Data augmentation iterates the following two steps:
\begin{enumerate}[(i)]
	\item Draw $(\mu_{ij})$ conditional on the data $\mathcal{Z}_n$ and the parameter $\nu$.
	\item Draw $\nu$ conditional on $(\mu_{ij})$.
\end{enumerate}
Once the algorithm has been run long enough, this gives approximate (dependent) samples from the posterior of $\nu$. We already know how to deal with step (ii); now we need to handle step (i).

Thus, keeping $\nu$ fixed, for each $i$ we want to compute the conditional distribution
$(\mu_{ij}:j=1,\ldots,m) \mid Z_i$, and furthermore, we want to be able to simulate from this distribution. In turn, this will immediately allow us to simulate $\mu_k$ conditional on the data $\mathcal{Z}_n$. Now, with $\pr(\cdot)$ referring to probability under the parameter $\nu$, it holds that
\begin{multline*}
\pr(\mu_{i1}=k_1,\ldots,\mu_{im}=k_m \mid Z_i = z_i) \\
 = \frac{1}{\pr(Z_i = z_i)} e^{-\Delta_i \sum_{j=1}^m \nu_j} \prod_{j=1}^m\frac{(\Delta_i\nu_j)^{k_j}}{k_j!} \ind\Big\{\sum_{j=1}^m j k_j = z_i\Big\}.
\end{multline*}
Knowledge of the normalising constant $\pr(Z_i = z_i)$ will not be needed in our approach.

In general, simulation from a discrete multivariate distribution is non-trivial; some general options are discussed in \cite{devroye86}, Chapter XI, Section 1.5, but are unlikely to work easily for a large $m$. We will take an alternative route and use the Metropolis-Hastings algorithm, see, e.g., Section 24.4 in \cite{wasserman04}. We start by observing that for a fixed $i$, the support of $\pr( \cdot \mid  Z_i = z_i)$ is precisely the set $\mathcal{S}_i$ of non-negative solutions $(k_1,\ldots,k_m)$ of the Diophantine equation
$
\sum_{j=1}^m j k_j = z_i.
$
The {\bf R} package {\bf nilde} (see \cite{nilde}) implements an algorithm from \cite{voinov97} that finds all such solutions for given integers $m$ and $z_i$. By Markovianity of the process $X$, we can simulate the vectors $(\mu_{i1},\ldots,\mu_{im})$  independently for each $i=1,\ldots,n$. If $z_i=0$ or $1$, there is only one solution to the Diophantine equation: the trivial solution $(0,\ldots,0)$ in the first case, and the solution $(1,0,\ldots,0)$ in the second case; for such $z_i$, no simulation is required, as $(\mu_{i1},\ldots,\mu_{im})$ is known explicitly. We thus only need to consider each $i\in\mathcal{I} = \{i:z_i \neq 0 \textrm{ or } 1\}$ in turn, and design a Metropolis-Hastings move on the set of the corresponding solutions $\mathcal{S}_i$. Fix once and for all an ordering of elements in $\mathcal{S}_i$ (this could be, e.g., lexicographic, or reverse lexicographic); we use the notation $|\mathcal{S}_i|$ for the cardinality of $\mathcal{S}_i$. Let $\mu = (\mu_{i1},\ldots,\mu_{im})$ be the current state of the chain, corresponding to the $\ell$th element $s_{\ell}$ of $\mathcal{S}_i$. A proposal $\mu^{\circ} = (\mu_{i1}^{\circ},\ldots,\mu_{im}^{\circ})$ is obtained as follows:

\begin{enumerate}[(i)]
	\item If $\ell = 1$, draw $\mu^{\circ}$ uniformly at random among the elements $\{s_2,s_{|\mathcal{S}_i|}\}$ of $\mathcal{S}_i$.
	\item If $\ell = |\mathcal{S}_i|$, draw  $\mu^{\circ}$ uniformly at random among the elements $\{s_1,s_{ |\mathcal{S}_i| -1 }\}$ of $\mathcal{S}_i$.
	\item If $\ell \neq 1$ or $|\mathcal{S}_i|$, draw  $\mu^{\circ}$ uniformly at random among the elements $\{s_{\ell-1},s_{\ell+1}\}$ of $\mathcal{S}_i$.

\end{enumerate}
Occasionally, one may want to propose another type of a move too.
\begin{enumerate}[(i)]
	\setcounter{enumi}{3}
	\item Draw $\mu^{\circ} = (\mu_{i1}^{\circ},\ldots,\mu_{im}^{\circ})$ uniformly at random from $\mathcal{S}_i$.
\end{enumerate}
The two proposals lead to reversible moves, and one may also alternate them with probabilities $\pi$ and $1-\pi$, e.g.\ $\pi=0.8$. The logarithm of the acceptance probability of a move from $(\mu_{i1},\ldots,\mu_{im})$ to $(\mu_{i1}^{\circ},\ldots,\mu_{im}^{\circ})$ is computed as
\[
\log A = \sum_{k=1}^m (\mu_{ik}^{\circ} - \mu_{ik}) \log(\Delta_i  \nu_k) + \sum_{k=1}^m \left \{ \log(\mu_{ik}!) - \log(\mu_{ik}^{\circ}!) \right\}.
\]
The move is accepted if $\log U \leq  \log A$ for $U$ an independently generated uniform random variate on $[0,1]$, and in that case the current state of the chain is reset to $(\mu_{i1}^{\circ},\ldots,\mu_{im}^{\circ})$. Otherwise the chain stays in $(\mu_{i1},\ldots,\mu_{im})$. 



\section{Simulation  examples}
\label{sec:examples}

In this section, we test performance of our approach in a range of representative simulation examples. For benchmarking, we use the frequentist plug-in estimator from \cite{buchmann04}. Two real data examples are given in Section \ref{sec:realdata}. Unless otherwise stated, we took $c=2$ and $a=0.01$ as hyperparameters in our prior specification. As can be seen from the update formulae for the Gibbs sampler, as long as $a$ is not taken too large, its precise value is not very influential on the posterior, given a reasonable sample size. The value $c=2$ ensures that the update step for $\beta_k$ has finite variance. At each step of  updating the imputed data for increment size $z$ we have chosen with probability $0.2$ to propose uniformly from all solutions to the Diophantine equation (for that particular value of $z$). 

We implemented our procedure in Julia, see \cite{julia}. The code and datasets for replication of our examples are available on GitHub\footnote{See \url{https://github.com/fmeulen/Bdd}} and Zenodo, see \cite{zenodo_bdd}.

\subsection{Uniform base distribution}\label{subsec:uniform}

This simulation example follows with some extensions that in \cite{buchmann04}. Let $\lambda_0 = 2$, and let $P_0$ be the discrete uniform distribution on $\{1,4,6\}$. We simulated data according to the following settings:
\begin{itemize}
  \item[(a)] $n=100$, $\Delta_i=1$ for $1\le i\le n$;
  \item[(b)] $n=500$, $\Delta_i=1$ for $1\le i\le n$ (the data under (a) are augmented with $400$ extra observations);
  \item[(c)] $n=500$, $\Delta_i=\operatorname{Unif}(0,2)$ for $1\le i\le n$.  
\end{itemize}

We set $m= \min(15,Z_{(n)})$, where $Z_{(n)} = \max_{1\leq i \leq n} Z_i.$ In all cases this led to $m=15$, as the value of $Z_{(n)}$ was equal to $30$, $35$ and $40$ for the simulated data under settings (a), (b) and (c), respectively.
The Gibbs sampler was run for $500{,}000$ iterations, of which the first $250{,}000$ were discarded as burn-in. From the remaining samples, the posterior mean and $2.5\%$ and $97.5\%$ percentiles were computed for each coefficient $\nu_k$. The results for the first $10$ coefficients are shown in Figure  \ref{fig:example1}. For comparison, the estimator from \cite{buchmann04} is also included in the figure. 
\begin{figure}
	\centering
		\captionsetup{width=0.8\textwidth, font=small}
		\includegraphics[trim=0.21cm 0cm 0cm 0cm, clip,width=0.8\textwidth]{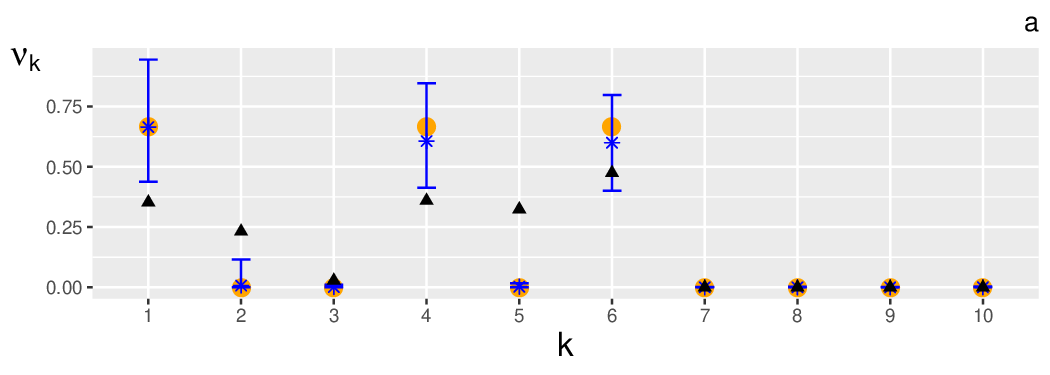}
		\includegraphics[trim=0.21cm 0cm 0cm 0cm, clip,width=0.8\textwidth]{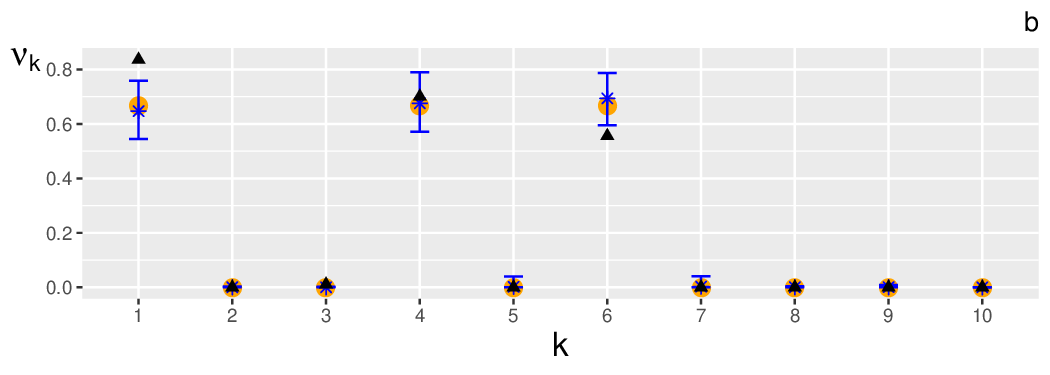}
		\includegraphics[trim=0.21cm 0cm 0cm 0cm, clip,width=0.8\textwidth]{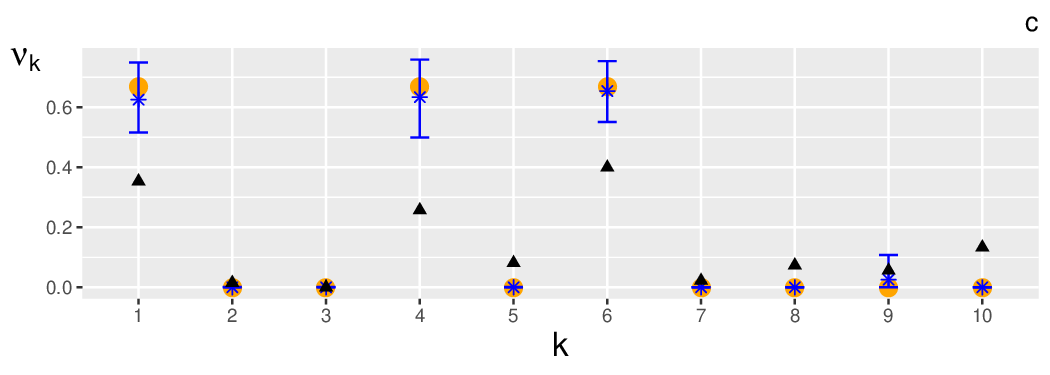}
	\caption{Simulation example from Section \ref{subsec:uniform}. In each figure, the horizontal axis gives the magnitudes of $\nu_k$, $k\in \{1,\ldots, 10\}$. The orange balls denote the true values, the black triangles the Buchmann-Gr\"{u}bel estimator. The blue crosses give the posterior means, whereas the vertical blue line segments represent (pointwise) $95\%$ credible intervals. The settings corresponding to (a), (b) and (c) are explained in the main text. Note the differences in vertical scale across the figures.
  \label{fig:example1}}
\end{figure}
For setting (b), traceplots of every $50$th iteration for a couple of coefficients $\nu_k$ are shown in Figure \ref{fig:example1-traceA2}. 
\begin{figure}
	\centering
	\captionsetup{width=0.8\textwidth, font=small}
	\includegraphics[trim=0.6cm 0cm 0cm 0cm, clip,width=0.8\textwidth]{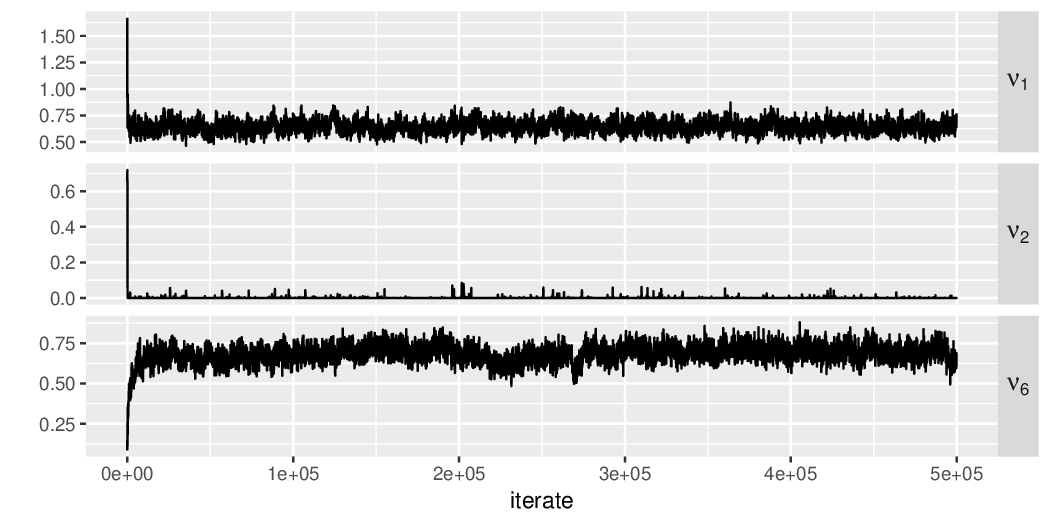}
		\caption{Traceplots for the simulation example from Section \ref{subsec:uniform} under setting (b). The posterior samples were subsampled, with every $50$th iteration kept. The displayed results are for parameters $\nu_1$, $\nu_6$ and $\nu_9$.
			\label{fig:example1-traceA2}}
\end{figure}
We measure the error of an estimate $\{\hat\nu_k\}$ by $\operatorname{Err}(\nu, \hat\nu) = \sum_{k=1}^{\infty} |\hat\nu_k - \nu_k|$. The errors are reported in Table \ref{table:unif}. In all  settings, for these particular realisations of the simulated data, the Bayesian procedure outperformed the truncated estimator from \cite{buchmann04}. For setting (c), the latter produces a poor result, as was to be expected, given that it is derived under the assumption $\Delta_i=1$ for all $i$. An advantage of the Bayesian procedure is the included measure of uncertainty, namely the credible intervals for $\nu_k$. On the other hand, for  the Buchmann-Gr\"{u}bel estimator it is hardly possible to derive confidence intervals via an asymptotic method, since the limiting distribution of the estimator is fairly complicated. Although not considered in the original publications \cite{buchmann03} and \cite{buchmann04}, a natural alternative is the bootstrap. A detailed examination of the performance of the latter and its comparison to that of the Bayesian method lies beyond the scope of the present paper. Indeed, any thorough study would require, on one hand, the asymptotic justification of bootstrap confidence intervals, and on another hand establishing frequentist coverage properties of our Bayesian procedure. In that respect, good performance of neither method is automatically warranted (see, e.g., \cite{vanderpas17} and \cite{vdv98}, Chapter~23). Here instead we opt for a numerical illustration, which is reported in Appendix \ref{app:bg:boot}.

{\small   
\begin{table}
\begin{center}
\captionsetup{width=0.8\textwidth, font=small}
\caption{Results for scenarios (a)--(c) from Section~\ref{subsec:uniform}.}
\begin{tabular}{l c c c}
\toprule
Simulation setting & (a) & (b) & (c)\\
\midrule
{Buchmann-Gr\"ubel estimator} & $1.40$ & $0.32$ & $1.44$ \\
{Posterior mean} & $0.15$ & $0.07$ & $0.12$ \\
\bottomrule
\end{tabular}
\label{table:unif}
\end{center}
\end{table}
}

\subsection{Geometric base distribution}
\label{subsec:geom}

The setup of this synthetic data example likewise follows that in \cite{buchmann04}. Assume $q$ is a geometric distribution with parameter $\alpha$, i.e.\ $q_k=(1-\alpha)^k \alpha$ for $0 < \alpha <1$, $k\in\N_0$. Then $\lambda=-\log\alpha$, and
\[
p_{k} = - \frac{(1-\alpha)^k}{k\log\alpha}, \quad k\in\mathbb{N}.
\]
Hence, $\nu_{k}=(1-\alpha)^k/k$.

We consider two simulation setups:
\begin{itemize}
\item[(a)] $n=500$, $\Delta_i=1$ for $1\le i \le n$ and  $\alpha=1/3$;
\item[(b)] $n=500$, $\Delta_i=1$ for $1\le i \le n$ and $\alpha=1/6$.
\end{itemize}
We set $m= \min(15,Z_{(n)})$ and ran the sampler according to the settings of Section \ref{subsec:uniform}. 
The results for both scenarios (a) and (b) are reported in Figure \ref{fig:example2}. In Table \ref{table:geom} we also report estimation errors in one simulation run. For this example and these generated data, the Bayesian procedure gives less precise point estimates than the Buchmann-Gr\"{u}bel method. Note that estimation error for $\alpha = 1/3$ is smaller than that for $\alpha = 1/6.$ This appears intuitive, as a smaller value of $\alpha$ corresponds to a larger value of $\lambda.$ The latter implies that on average each $Z_i$ is a superposition of a larger number of jumps, which renders the decompounding problem more difficult. However, this argument is hard to formalise.

{\small
\begin{table}
\captionsetup{width=0.8\textwidth, font=small}
\caption{Results for scenarios (a)--(b) from Section~\ref{subsec:geom}.}
\begin{center}
	\begin{tabular}{l c c }
		\toprule
		Simulation setting & (a) & (b) \\
		\midrule
		{Buchmann-Gr\"ubel estimator} & $0.28$ & $0.60$ \\
		{Posterior mean} & $0.52$ & $1.05$ \\
		\bottomrule
	\end{tabular}
\label{table:geom}
\end{center}
\end{table}
}

\begin{figure}
	\centering
	\captionsetup{width=0.8\textwidth, font=small}
		\includegraphics[trim=0.21cm 0cm 0cm 0cm, clip,width=0.8\textwidth]{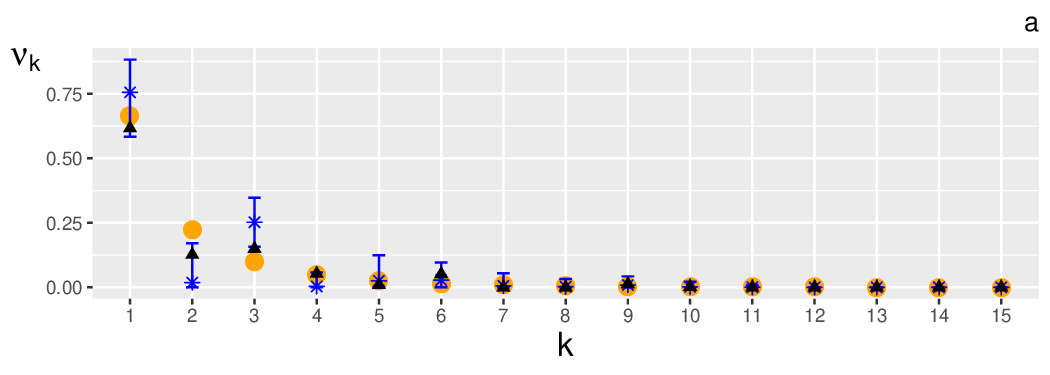}\\
		\includegraphics[trim=0.21cm 0cm 0cm 0cm, clip,width=0.8\textwidth]{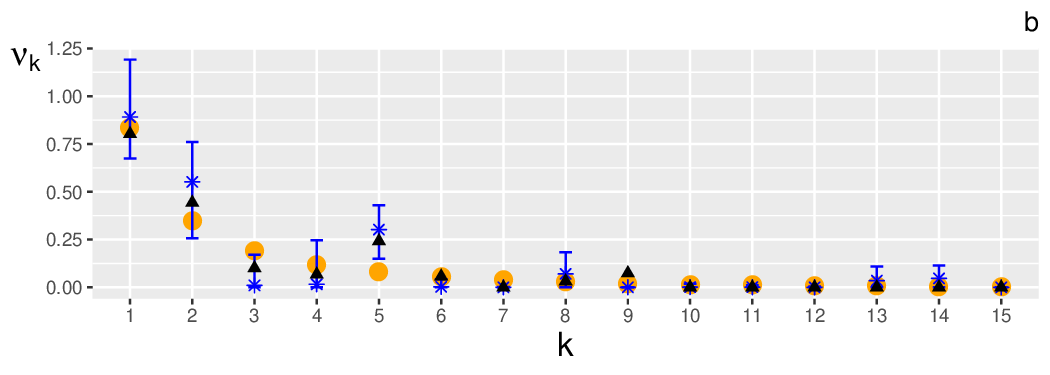}
	\caption{Simulation example from Section \ref{subsec:geom}.  Settings (a) and (b)  correspond to the true jump distributions $\geom(1/3)$ and $\geom(1/6),$ respectively. The horizontal axis gives the magnitudes of $\nu_k$, $k\in 
	\{1,\ldots, 15\}$. The orange balls denote the true values, the black triangles the Buchmann-Gr\"{u}bel estimator. The blue crosses give the posterior means, whereas the vertical blue line segments represent (pointwise) $95\%$ credible intervals.
  \label{fig:example2}}
\end{figure}


\subsection{Monte Carlo study}\label{subsec:mc}
For a more thorough comparison of the Buchmann-Gr\"{u}bel estimator and our Bayesian method, we performed a small Monte Carlo experiment. We considered two settings:
\begin{enumerate}
  \item The setting from Section \ref{subsec:uniform} with $n=250$. We took $m=\min(15, Z_{(n)})$.   
  \item The setting from Section \ref{subsec:geom} with $\alpha=1/3$. We took $m=\min(20, Z_{(n)})$.    
\end{enumerate}
In both cases we assumed $\Delta_i=1$ for all $1\le i \le n$. The number of Monte Carlo repetitions was taken equal to $50$. We took $400{,}000$ MCMC iterations and discarded the first half of these as burn-in. In Figure \ref{fig:mcresults} we give a graphical display of the results by means of boxplots of the errors. Here, as before, if the true values are denoted by $\nu_k$ and the estimate within a particular simulation run  by $\hat\nu_k$, the error is defined by  $\operatorname{Err}(\nu, \hat\nu) = \sum_{k=1}^{\infty} |\hat\nu_k - \nu_k|$ (we truncated the infinite summation to $50$). The results agree with our earlier findings, in that there is no clear ``winner'' in the comparison. Note that for the setting (ii) we considered both $c=2$ and $c=0.01$ in the prior specification. Both values give similar performance of the Bayesian method. This provides insight into sensitivity of our results with respect to the prior specification. A minor difference between the middle and righmost panel of Figure \ref{fig:mcresults} may be attributed to Monte Carlo error: the $50$ simulated datasets on which these panels are based are not the same.  Note that the prior promotes sparsity, and in that respect it is not surprising it does better when the true data-generating L\'evy measure is sparse. 
\begin{figure}
	\centering
	\captionsetup{width=0.8\textwidth, font=small}
		\includegraphics[trim=0.6cm 0cm 0cm 0cm, clip,width=0.8\textwidth]{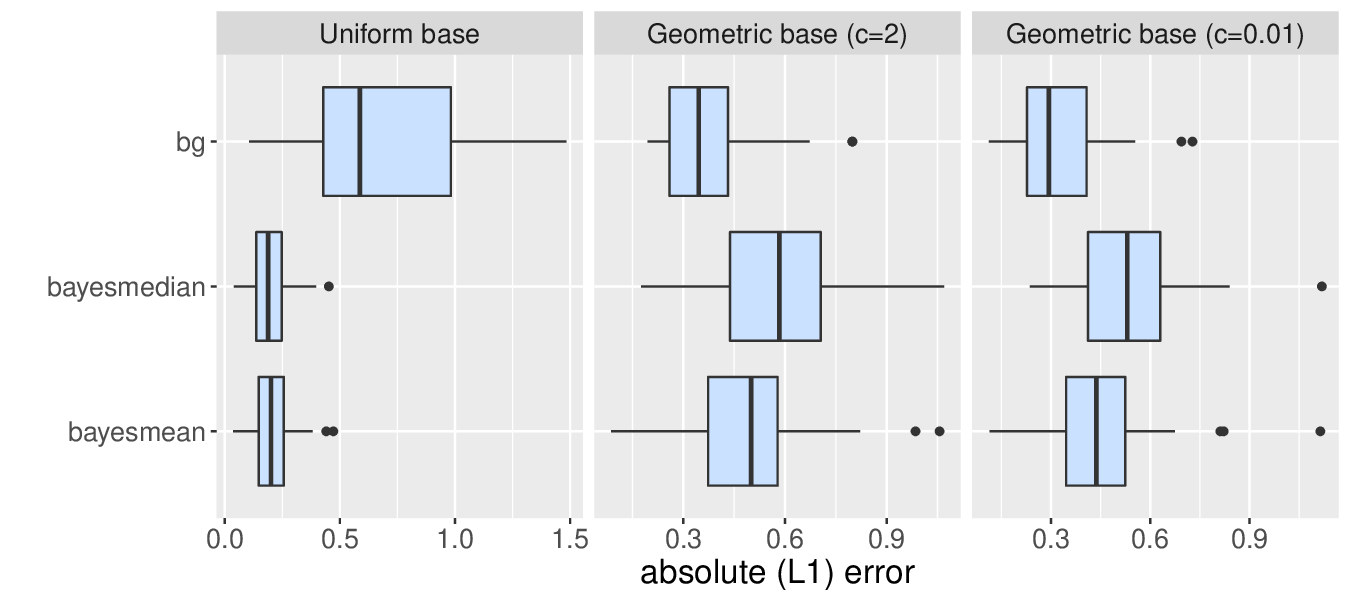}
	\caption{Monte Carlo study from Subsection \ref{subsec:mc} comparing the Buchmann-Gr\"{u}bel estimator and the Bayesian method proposed in this paper. In this figure ``bg'' refers to the Buchmann-Gr\"{u}bel estimator, while  ``bayesmedian'' and ``bayesmean'' refer to the Bayesian method, where either the median or mean was used as a point estimator for each $\nu_i$. The leftmost panel corresponds to the setting (i), whereas the other two panels to the setting (ii). In the latter we used both $c=2$ and $c=0.01$ in the prior specification.  \label{fig:mcresults}}
\end{figure}

\subsection{Computing time}
In terms of computational effort, the time it takes to evaluate the Buchmann-Gr\"{u}bel estimator is negligible compared to our algorithm for sampling from the posterior. This is not surprising, as that frequentist estimator relies on a plug-in approach, whereas in our case an approximation to the posterior is obtained by MCMC simulation. However, if proper uncertainty quantification is desired, then the Bayesian method is advantageous in the sense that it does not solely produce a point estimate.

Note that the proposed MCMC scheme requires determination of the solutions to the Diophantine equation $\sum_{j=1}^m j k_j = z$ for all unique values  $z$ in the observation set. For moderate values of $z$, say $z\le 30,$ this is rather quick, but for large values of $z$ the computing time increases exponentially, as does the amount of the allocated memory. The computing time of each Metropolis-Hastings step is then very small, but we potentially need a very large number of iterations to reach stationarity. The latter is caused firstly  by the fact that at a particular iteration, our proposals for $\mu_{ij}$ do not take  into account the current values of $\nu_1,\ldots, \nu_m$; secondly, the size of the state space that needs to be explored  increases exponentially with $m$. 

\section{Real data examples}\label{sec:realdata}
\subsection{Horse kick data}\label{subsec:horse}

To further illustrate our procedure, we will use the von Bortkewitsch data on the number of soldiers in the Prussian cavalry killed by horse kicks (available by year and by cavalry corps); this example was also employed in \cite{buchmann03}. Each observation is an integer from $0$ to $4$, giving the number of deaths for a given year and a given cavalry corps, with overall counts reported in Table \ref{table:bort}. The data are extracted from the table on p.~25 in \cite{bortkewitsch98}. Note that von Bortkewitsch corrects for the fact that the Guards and I, VI and XI cavalry corpses of the Prussian army had a different organisation from other units, and justifiably omits the corresponding counts from consideration.

It has been demonstrated by von Bortkewitsch that the Poisson distribution fits the horse kick data remarkably well. Assuming instead that observations follow a compound Poisson distribution is a stretch of imagination, as that would correspond to a horse running amok and killing possibly more than one soldier in one go. Nevertheless, this example provides a good sanity check for our statistical procedure.  

{\small
\begin{table}
	\begin{center}
		\captionsetup{width=0.8\textwidth, font=small}
		\caption{Data on the number of soldiers in the Prussian cavalry killed by horse kicks. See the table on p.~25 in \cite{bortkewitsch98}.}
		\begin{tabular}{lrrrrr}
			\toprule 
			{Deaths} & { $0$} & { $1$} & { $2$} 
			& { $3$} & { $4$}  \\
			\midrule
			{Counts}     & $109$  & $65$  & $22$   & $3$ & $1$ \\
			\bottomrule
		\end{tabular}
		\label{table:bort}
	\end{center}
\end{table}
}

The estimation results are graphically depicted in Figure \ref{fig:horsekick}. Clearly, point estimates of both methods are in agreement and lend support to the Poisson model for this dataset. 

\begin{figure}
	\centering
	\captionsetup{width=0.8\textwidth, font=small}
		\includegraphics[trim=0.21cm 0cm 0cm 0cm, clip,width=0.8\textwidth]{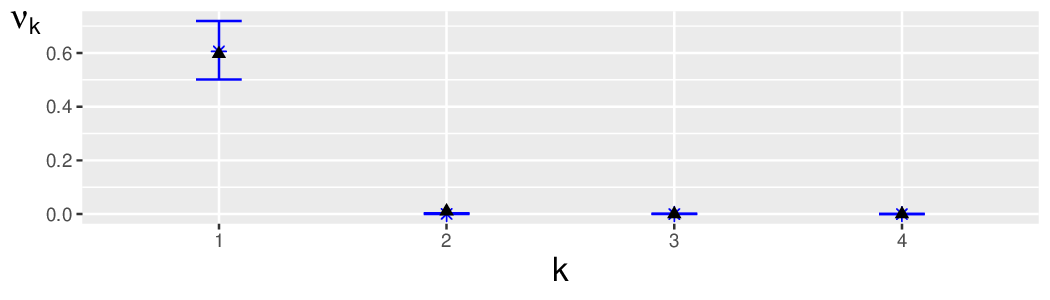}
	\caption{Estimation for the horse kick data from Subsection \ref{subsec:horse}. The horizontal axis gives the magnitudes of $\nu_k$, $k\in 
	\{1,\ldots, 4\}$. The black triangles denote the Buchmann-Gr\"{u}bel estimator, the blue crosses give the posterior means, whereas the vertical blue line segments represent (pointwise) $95\%$ credible intervals.
  \label{fig:horsekick}}
\end{figure}

\subsection{Plant data}\label{subsec:plant}

Our second real example is the one used in \cite{buchmann04}. Consider the data in Table \ref{table:plant}, taken from \cite{evans53}. The data were collected as follows: the area was divided into plots of equal size and in each plot the number of plants was counted; the number of plants in each plot ranges from $0$ to $12$. The second row of Table \ref{table:plant} gives the counts of plots containing a given number of plants; thus, there were $274$ plots that contained no plant, $71$ that contained $1$ plant, etc. It is customary in the ecological literature to model such count data as i.i.d.\ realizations from a compound Poisson distribution. Thus, e.g., \cite{neyman39} advocated the use of a Poisson base distribution in this context; another option here is a geometric base distribution. Given existence of several distinct modelling possibilities, performing an exploratory non-parametric  analysis appears to be a sensible strategy.

The estimation results are graphically depicted in Figure \ref{fig:plant}. There are some small differences between the posterior mean and the Buchmann-Gr\"ubel estimate, but overall they are very similar. 
\begin{table}
	\begin{center}
		\captionsetup{width=0.8\textwidth, font=small}
		\caption{Plant population data from \cite{evans53}.}
		\begin{tabular}{lrrrrrrrrrrrrr}
			\toprule 
			{Plants} & $0$ & $1$ & $2$ 
			& $3$ & $4$ & $5$ & $6$ & $7$ & $8$ & $9$ & $10$ & $11$ & $12$  \\
			\midrule
			{Counts}     & $274$  & $71$  & $58$   & $36$ & $20$ & $12$ & $10$ & $7$ & $6$ & $3$ & $0$ & $2$ & $1$ \\
			\bottomrule
		\end{tabular}
		\label{table:plant}
	\end{center}
\end{table}

\begin{figure}
	\centering
	\captionsetup{width=0.8\textwidth, font=small}
	\includegraphics[trim=0.21cm 0cm 0cm 0cm, clip,width=0.8\textwidth]{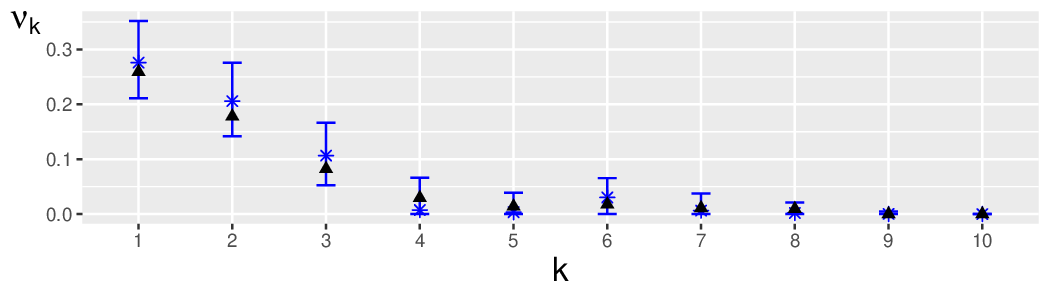}
	\caption{Estimation results for the plant data from Subsection \ref{subsec:plant}. The horizontal axis gives the magnitudes of $\nu_k$, $k\in 
	\{1,\ldots, 10\}$. The black triangles denote the Buchmann-Gr\"{u}bel estimates, the blue crosses give the posterior means, whereas the vertical blue line segments represent (pointwise) $95\%$ credible intervals.
  \label{fig:plant}}
\end{figure}

\section{Frequentist asymptotics}
\label{sec:asymptotics}

In this section we assume that the observation times $\{t_i\}$ are equidistant: $t_i = i, i=1,\ldots,n$. To evaluate our Bayesian method from a theoretical point of view, we will verify that it is consistent, and we will establish the rate at which the posterior contracts around the `true', data-generating L\'evy measure $\nu_0$; see \cite{ghosal17} for a thorough treatment of Bayesian asymptotics from the frequentist point of view. From now on the subscript $0$ in various quantities will refer to the data-generating distribution.

Our strategy consists in proving that the posterior contraction rate for $\nu_0$, given the sample $\mathcal{Z}_n=(Z_1,\ldots, Z_n)$, can be derived from the posterior contraction rate for $q_0$ given $\mathcal{Z}_n$, which is mathematically easier since $Z_1,\ldots, Z_n$ is a sequence of independent and identically distributed random variables with distribution $q_0$. We therefore effectively avoid dealing directly with the inverse nature of the problem of estimating $p_0$.

The  prior we consider in this section is defined  as follows:
\begin{itemize}
  \item Endow the rate $\lambda$ of the Poisson process with a prior distribution.
  \item Independently, endow the vector $(p_1,\ldots, p_{m})$ with a Dirichlet distribution with parameter $(\alpha_1,\ldots,\alpha_m)$.
  \item Set a priori $p_k=0$ for all $k > m$. 
\end{itemize}
This is a somewhat simplified version of the prior we used in Section \ref{sec:algorithm}, which allows us to concentrate on essential features of the problem, without need to clutter the analysis with extra and unenlightening technicalities. Also remember the well-known relationship between the Gamma and Dirichlet distributions: if $\xi_1,\ldots,\xi_m$ are independent Gamma distributed random variables, $\xi_i \sim \operatorname{Gamma}(\alpha_i,1)$, then for $\eta_i = \xi_i / \sum_{j=1}^m \xi_j$, the vector $(\eta_1,\ldots,\eta_m)$ follows the Dirichlet distribution with parameter $(\alpha_1,\ldots,\alpha_m)$; furthermore, we have that $\sum_{j=1}^m \xi_j \sim \operatorname{Gamma}\left(\sum_{j=1}^m \alpha_j,1\right)$, and $\eta_i$ are independent of $\sum_{j=1}^m \xi_j$.

In our asymptotic setting, we will make $m=m_n$ dependent on $n$ and let $m_n\rightarrow\infty$ at a suitable rate as $n\rightarrow\infty$.

Recall that we write $Q=(q_k)_{k\in\mathbb{N}_0}$ for $q_k=Q(\{k\})$. Let $\scr{Q}$ denote the collection of all probability measures supported on $\mathbb{N}$.
\begin{thm}\label{thm:postcontractionrate}
	Suppose there exists  $\underline\alpha$, such that $0<\underline{\alpha} \le \alpha_i \le 1$ for all $1\le i \le m_n$. Suppose $\lambda \sim \operatorname{Gamma}(a,b)$ with $a\in (0,1]$ and that $\nu_0$ has a compact support. 
Then, for any $\gamma >1$,
\[  \Pi_n\left( \| \nu - \nu_0 \|_1 \geq \frac{\log^{\gamma} n}{\sqrt{n}} \,\middle\vert\, \mathcal{Z}_n \right) \to 0\] in $Q_0^n$-probability, as $n\to \infty$.
\end{thm}

\begin{rem}
	Note that since the support of $\nu_0$ is not assumed to be known, our CPP model is still naturally non-parametric. The assumption of the compact support of $\nu_0$ does not cover the simulation example of Section \ref{subsec:geom}. However, its relaxation appears to pose very difficult technical challenges and is not attempted in this work.  
\end{rem}

The remainder of this section is devoted to the proof of Theorem \ref{thm:postcontractionrate}.

\subsection{Basic posterior inequality via the stability estimate}
A key step of the proof of Theorem~\ref{thm:postcontractionrate} is the stability estimate in Equation \eqref{eq:tv} below, which bounds the total variation distance between the L\'evy measures $\nu, \nu^{\prime}$ in terms of the total variation distance between the corresponding compound distributions $q,q^{\prime}$.

In principle, it is conceivable that the Panjer recursion should allow one to bound probability distances between $P$-probabilities via distances between $Q$-probabilities; we call such a bound a stability estimate. Nevertheless, explicit as the equations of the Panjer recursion are, they are still somewhat unwieldy for that purpose. Hence we will use another inversion formula from \cite{buchmann03}, that will lead to the stability estimate we are after.

First we introduce some notation, and also recall a few useful facts summarised in \cite{buchmann03}. The space of absolutely summable sequences is defined as
$\ell_1 \coloneqq \left\{ a\in\rr^{\mathbb{N}_0}: \sum_{j=0}^{\infty} |a_j| < \infty \right\}$,
with a norm given by
$\|a\|_1 = \sum_{j=0}^{\infty} |a_j|$.
For probability vectors $a,b$, the norm $\|a-b\|_1$ is (twice) the total variation distance between $a$ and $b$. For any $a,b\in\ell_1$, we have the inequality
\begin{equation}
\label{eq:ineq:conv}
\|a\ast b\|_1 \leq \|a\|_1 \|b\|_1,
\end{equation}
where $\ast$ denotes convolution of $a$ and $b.$ We define a mapping $a\mapsto\exp(a)$ from $\ell_1$ into $\ell_1$ via
\[
\exp(a)=\sum_{j=0}^{\infty} \frac{a^{\ast j}}{j!}.
\]
The exponential has the following two useful properties:
\[
\exp(a+b)=\exp(a)\ast\exp(b), \quad a,b\in\ell_1,
\]
and
\[
\exp(a)=\exp(b) \Longrightarrow a=b, \quad a,b\in\ell_1.
\]
We define a sequence $\delta_0=(\delta_{0,k})_{k\in\mathbb{N}_0}$, such that $\delta_{0,0}=1$ and its all other entries are equal to zero. Then, using the above properties of the exponential, we can write  concisely the compounding mapping in \eqref{eq:comp} in terms of convolutions of infinite sequences: $q=\exp(\lambda(p-\delta_0))$. Its convolution inverse, i.e.\ $q^{\ast(-1)}$ such that $q^{\ast(-1)} \ast q = \delta_0$, is given by $r=q^{\ast(-1)}=\exp(-\lambda(p-\delta_0))$. Note that $r\in\ell_1$. We have the following recursive expressions
\[
r_0 = \frac{1}{q_0}, \quad r_k = - \frac{1}{q_0} \sum_{j=1}^k q_j r_{k-j}, \quad k\in\mathbb{N}.
\]

\begin{lemma}\label{lemma:tv}
Let $q,q^{\prime}$ correspond to two pairs $(\lambda,p)$ and $(\lambda^{\prime},p^{\prime})$, respectively (and $r$ correspond to $q$, i.e.\ the pair $(\lambda,p)$). 
Then, in accordance with the notation introduced above and provided that $\|q^{\prime}-q\|_1<\|r\|_1^{-1}$, it holds that
\begin{equation}
\label{eq:tv}
\| \nu^{\prime} - \nu \|_1 = \| \lambda^{\prime}p^{\prime} - \lambda p \|_1 \leq \frac{\|r\|_1 \| q^{\prime} - q \|_1}{1 - \|r\|_1 \| q^{\prime} - q \|_1 }.
\end{equation}
\end{lemma}
\begin{proof}
The result is a direct consequence of Lemma 3 in \cite{buchmann03}, which states that 
\[
(\lambda^{\prime}-\lambda)\delta_0 + \lambda p - \lambda^{\prime}p^{\prime} = \sum_{j=1}^{\infty} \frac{1}{j}(r\ast(q-q^{\prime}))^{\ast j}
\]
whenever $\|q^{\prime}-q\|_1<\|r\|_1^{-1}$.
Taking the $\|\cdot\|_{1}$-norm on both sides and some elementary bounding via \eqref{eq:ineq:conv} imply that
$$|\lambda^{\prime}-\lambda|+ \| \lambda^{\prime}p^{\prime} - \lambda p \|_1 \leq \frac{\|r\|_1 \| q^{\prime} - q \|_1}{1 - \|r\|_1 \| q^{\prime} - q \|_1 },$$
and thus Equation \eqref{eq:tv} follows.
\end{proof}
We will use Equation \eqref{eq:tv} to establish the key inequality for the posterior measure $\Pi(\cdot \mid \mathcal{Z}_n)$. We recall once again that the subscript $0$ refers to `true', data-generating quantities.

\begin{prop}\label{prop:contr}
For any prior $\Pi$ on $\nu$, for any $\varepsilon\in (0,1] $ and for any $n\geq 1$, the following posterior inequality holds:
$$\Pi\left( \| \nu - \nu_0 \|_1 \geq \varepsilon \,\middle\vert\, \mathcal{Z}_n \right)\leq 2 \Pi \left( \| q - q_0 \|_1 \geq \frac{\varepsilon}{2\|r_0\|_1} \,\middle\vert\,  \mathcal{Z}_n \right).$$
\end{prop}

\begin{proof}
Write $\{ \nu: \| \nu - \nu_0 \|_1 \geq \varepsilon \}$ as a union of the sets
\[\left\{ \nu: \| \nu - \nu_0 \|_1 \geq \varepsilon \right\} \cap \left\{ \nu: \|r_0\|_1 \| q - q_0 \|_1 < 1/2 \right\}\]
and
\[ \left\{ \nu: \| \nu - \nu_0 \|_1 \geq \varepsilon \right\} \cap \left\{ \nu: \|r_0\|_1 \| q - q_0 \|_1 \geq  1/2 \right\}.\]
Thanks to Lemma \ref{lemma:tv}, the set
\[
\left\{ \nu: \| \nu - \nu_0 \|_1 \geq \varepsilon \right\} \cap \left\{ \nu: \|r_0\|_1 \| q - q_0 \|_1 < 1/2 \right\}
\]
is a subset of $\{ \nu: \| q - q_0 \|_1 \geq {\varepsilon}/(2\|r_0\|_1) \}$. The proof is concluded by observing that $\left\{ \nu: \| \nu - \nu_0 \|_1 \geq \varepsilon \right\} \cap \left\{ \nu: \|r_0\|_1 \| q - q_0 \|_1 \geq  1/2 \right\}$ is a subset of $\{ \nu: \| q - q_0 \|_1 \geq {\varepsilon}/(2\|r_0\|_1) \}$, too, since $\varepsilon\leq 1$. 
\end{proof}
In general, stability estimates like the one in Equation \eqref{eq:tv} are unknown in the literature on L\'evy processes. Consequently, studying Bayesian asymptotics for L\'evy models, even in the CPP case, necessitates the use of very intricate arguments under restrictive assumptions; see, e.g., \cite{nickl17}.

\subsection{Proof of Theorem \ref{thm:postcontractionrate}}
The usefulness of Proposition \ref{prop:contr} above lies in the fact that the posterior contraction rate in the inverse problem of estimating the L\'evy measure $\nu_0$ from indirect observations $\mathcal{Z}_n$ can be now deduced from the posterior contraction rate in the direct problem of estimating the compound distribution $q_0$, which is easier (observe that $r_0$ is determined by $\nu_0$ and is therefore fixed in the proofs). The general machinery developed in \cite{ghosal00} can be applied to handle the latter, and also several inequalities obtained in \cite{gugu15} are useful in that respect.
In particular, we make use of the following inequality for the Hellinger distance, 
\begin{equation}
\label{eq:HL}
h(q_{\la, p}, q_{\la^{\prime}, p^{\prime}}) \le \sqrt{\lambda} h(p,p^{\prime}) +|\sqrt{\lambda} - \sqrt{\lambda^{\prime}}|,
\end{equation}
Cf.\ Lemma 1 in \cite{gugu15}. 
To ease our notation, in the sequel we will often write  $q$ and $q^{\prime}$ instead of $q_{\lambda, p}$ and $q_{\lambda^{\prime}, p^{\prime}},$ respectively.

Denote
\[ KL(q_0,q) = Q_0\left(\log \frac{q_0}{q}\right), \quad V(q_0,q) = Q_0\left(\log \frac{q_0}{q}\right)^2. \]
Another two inequalities we will use are the following: let $\lambda, \lambda_0 \in [\underline{\lambda}, \overline{\lambda}].$ Then there exists a positive constant $\overline{C}$, such that 
\begin{equation}\label{eq:transfer-qtolambdaandp}
\begin{split}  KL(q_0, q) \le \overline{C}\left(KL(p_0,p) +|\lambda_0- \lambda|^2\right), \\
 V(q_0, q) \le \overline{C}\left(V(p_0,p)+ KL(p_0,p) +|\lambda_0- \lambda|^2\right); \end{split} 
 \end{equation}
cf.~equations (14) and (15) in Lemma 1 in \cite{gugu15}.

These three inequalities can be obtained by adjustment of the arguments used in \cite{gugu15}. However, we opted to give their direct proofs in Lemma \ref{KLV} from Appendix ~\ref{app:lemmas} under slightly weaker conditions than required in \cite{gugu15}.

Our proof of Theorem \ref{thm:postcontractionrate} proceeds via verification of the conditions for posterior contraction in Theorem 2.1 in \cite{ghosal00}.  In our setting, the latter result reads  as follows.
\begin{thm}\label{thm:ggvdv}
Assume $\scr{Z}_n =(Z_1,\ldots, Z_n),$ where $Z_1,\ldots, Z_n$ are independent and identically distributed with distribution $q_0$. Let $h$ denote the Hellinger metric on $\scr{Q}$, a collection of all measures with support in $\mathbb{N}$. 
Suppose that for a sequence $\{\eps_n\}$ with $\eps_n \to 0$ and $n\eps_n^2\to \infty$, a constant $C>0$ and sets $\scr{Q}_n\subset \scr{Q}$, we have
\[ \log N(\eps_n, \scr{Q}_n, h) \le n\eps_n^2,\]
\[ \Pi_n(\scr{Q} \setminus \scr{Q}_n) \le \exp\left(-n\eps_n^2(C+4)\right), \]
\[ \Pi_n\left( q\colon KL(q_0,q) \le \eps_n^2,\, V(q_0,q) \le \eps_n^2\right) \ge \exp\left(-C n\eps_n^2\right). \]
Then, for sufficiently large $M>0$, we have that $\Pi_n \left(Q\colon h(q,q_0)\ge M\eps_n \mid \scr{Z}_n \right) \to 0$ in $Q_0^n$-probability. 
\end{thm}

We will now verify the three conditions of this theorem, which we refer to as the entropy condition, the remaining mass condition, and the prior mass condition, respectively.
To that end, 
fix strictly positive sequences $\{\underline\Lambda_n\}$, $\{\overline\Lambda_n\}$, and 
define the sieves
\[ \scr{Q}_n =\left\{ q_{\la,p} \colon \la \in [\underline{\Lambda}_n, \overline\Lambda_n],\: \operatorname{supp} p \subseteq \{1,\ldots,m_n\} \right\}. \]

\subsubsection{Entropy}
We start with bounding the entropy of the sieve $\scr{Q}_n$ for $h$-balls of radius $\eps_n$. 
\begin{lemma}
	\label{lem:entropy}
Assume that as $n\to\infty,$
\begin{equation}
\label{eq:nlarge}
m_n\rightarrow\infty, \quad \epsilon_n \rightarrow 0, \quad \underline{\Lambda}_n \rightarrow 0, \quad \overline{\Lambda}_n \rightarrow \infty.
\end{equation}
Then 
\begin{equation}\label{eq:entropy}
\log N(\epsilon_n, \scr{Q}_n, h) \lesssim m_n \left\{ \log (m_n) + \log (\overline{\Lambda}_n) +  \log\left( \frac{1}{\eps_n} \right) \right\} + \log \left(\frac{1}{\underline{\Lambda}_n} \right). \end{equation}
\end{lemma}

\begin{proof}
For $\la, \la' \ge \underline{\Lambda}_n$,
\[ |\sqrt{\la}-\sqrt{\la'}| = \frac{|\la-\la'|}{\sqrt{\la}+\sqrt{\la'}} \le \frac1{2\sqrt{\underline{\Lambda}_n}} |\la -\la'|. \]
Furthermore, from Section 3.3 in \cite{pollard_2001},
\[   	h(p, p')   \leq \sqrt{ \| p - p^{\prime}\|_1 }  \leq \sqrt{ m_n \| p-p' \|_{\infty} }. \]
Combining the preceding two displays and Equation \eqref{eq:HL}, we get 
\[  h(q_{\la, p}, q_{\la', p'}) \le  \sqrt{\overline{\Lambda}_n m_n \|p-p'\|_{\infty}} + \frac1{2\sqrt{\underline{\Lambda}_n}} |\la -\la'|. \]
Hence, if
\[ \|p-p'\|_{\infty} \le \frac{\eps_n^2}{4 \overline{ \Lambda}_n m_n}, \quad  |\la-\la'| \le \sqrt{\underline{\Lambda}_n} \eps_n, \]
then  the Hellinger distance between $q_{\la, p}$ and $q_{\la', p'}$ is bounded by $\eps_n$. 
To cover $[\underline{\Lambda}_n, \overline\Lambda_n]$, we need at most $\lfloor \frac{\overline{\Lambda}_n}{2\eps_n \sqrt{\underline{\Lambda}_n}}\rfloor +1$ intervals of length $2 \sqrt{\underline{\Lambda}_n} \eps_n$.
To cover discrete distributions with support in $\{1,\ldots,m_n\}$, we need at most 
\[ \left(  \left\lfloor \frac{2\overline{\Lambda}_n m_n }{\eps_n^2} \right\rfloor +1  \right)^{m_n}\]
$L_{\infty}$-balls of radius $\eps_n^2 / (4\overline{\Lambda}_n m_n)$. Under assumption \eqref{eq:nlarge}, the summand $1$ in the above display is asymptotically negligible and can be omitted. In that case, the number of $h$-balls that we need to cover $\scr{Q}_n$ is of order
\[ \left( \frac{\overline{\Lambda}_n m_n }{\eps_n^2} \right)^{m_n}\times  \frac{\overline{\Lambda}_n}{\eps_n \sqrt{\underline{\Lambda}_n}} . \]
Taking the logarithm and next a straightforward rearrangement of the terms gives the statement of the lemma. \end{proof}

\subsubsection{Remaining prior mass}
Now we will derive an inequality for the remaining prior mass.

\begin{lemma}\label{lem:remainingmass} For $\lambda\sim \operatorname{Gamma}(a,b)$ with $0<a \leq 1$,
	\[ \Pi_n(\scr{Q}\setminus \scr{Q}_n) \lesssim {\overline{\Lambda}}_n^{a - 1} e^{- b \overline\Lambda_n} +  \underline{\Lambda}_n .\]
\end{lemma}

\begin{proof}
We have (with a slight abuse of notation)
	\[ \Pi_n(\scr{Q}\setminus \scr{Q}_n) = \Pi_n\left( [\overline{\Lambda}_n, \infty)\right) +\Pi_n\left( [0,\underline{\Lambda}_n)\right).\]
Now,
\[
\Pi_n( \lambda \geq \overline{\Lambda}_n ) = \frac{ b^a }{\Gamma(a) } \int_{ \overline{\Lambda}_n }^{\infty} \lambda^{a - 1} e^{-b \lambda }\dd\lambda \lesssim \overline{\Lambda}_n^{a - 1} e^{-b \overline\Lambda_n}.
\]
Furthermore,
\[
\Pi_n\left( [0,\underline{\Lambda}_n)\right) = \frac{b^a}{\Gamma(a)  } \int_0^{ \underline{\Lambda}_n } \lambda^{a - 1} e^{- b\lambda}\dd\lambda \lesssim \underline{\Lambda}_n^a.
\]
The proof is concluded.
\end{proof}

\subsubsection{Prior mass}
Finally, we lower bound the prior mass in a small Kullback-Leibler neighbourhood of the data-generating compound distribution $q_0$. Define the function  $g\colon (0,\infty) \times (0,1) \to (0,\infty)$ by \[g(\eps,c) = C \frac{\eps^2}{2[\log(e/c)]^2},\] where $C$ is the constant appearing in the statement of Lemma \ref{lem:KLVtoL1} below.
\begin{lemma}\label{lem:priormass}
Assume that
\begin{enumerate}
  \item there exists  $\underline\alpha$, such that $0<\underline{\alpha} \le \alpha_i \le 1$ for all $1\le i \le m_n$;
  \item strictly positive sequences $\underline{p}_n \rightarrow 0$, $\eps_n\rightarrow 0$ and $m_n\rightarrow\infty$ satisfy the inequalities $m_n  g(\eps_n,\underline{p}_n)<1$ and $\underline{p}_n <g(\eps_n,\underline{p}_n)^2.$
\end{enumerate}
Define 
\[ B_n(\eps) = \left\{q \in \scr{Q}_n \colon  KL(q_0,q) \le \eps^2,\: V(q_0,q) \le \eps^2\right\}. \]
Then 
\begin{multline*} \Pi_n(B_n(\eps_n))  \gtrsim \Pi_n\left(  |\lambda_0-\lambda| \le \widetilde\eps_n\right)\\ \times 
	\Gamma\left(\sum_{i=1}^{m_n} \alpha_i\right)  \exp\left(-m_n\log(1/(g(\widetilde\eps_n,\underline{p}_n)^2-\underline{p}_n)) - m_n \log(1/\underline{\alpha})\right).
\end{multline*}
Here $\widetilde\eps_n = \eps_n/\sqrt{3\overline{C}}$, with a constant $\overline{C} > 0 $ not depending on $n$. 
\end{lemma}
\begin{proof}
Define 
\begin{multline*}
 \widetilde{B}_n(\eps) = \Big\{(\lambda, p)\colon \lambda \in [\underline\Lambda_n, \overline\Lambda_n],\: \min_{1\le i \le m_n} p_i \ge \underline{p}_n, \operatorname{supp} p \subseteq\{1,\ldots,m_n\}, \\  KL(p_0,p) \le \eps^2,\: V(p_0,p) \le \eps^2,\: |\lambda_0-\lambda| \le \eps\Big\} .\end{multline*}
For all $n$ large enough and $\eps$ small, we have $\{ \lambda: |\lambda_0-\lambda| \le \eps  \} \subseteq [\underline\Lambda_n, \overline\Lambda_n].$ Then by inequalities in Lemma \ref{KLV},  $\widetilde{B}_n(\eps)\subset B_n\left(\sqrt{3\overline{C}}\eps\right)$, with a constant $\overline{C}$ that can be taken the same for all large enough $n$; see the arguments in Section 4.2 in \cite{gugu15}. Hence, using the a priori independence of $p$ and $\lambda$,
\[ \Pi_n(B_n\left(\eps_n\right)) \ge \Pi_n\left(\widetilde{B}_n\left(\widetilde\eps_n\right)\right)  = \Pi_n\left(   |\lambda_0-\lambda| \le \widetilde\eps_n\right) \times U_n, \] where 
\[ U_n=  \Pi_n\left( \left\{p \colon  KL(p_0, p) \le \widetilde\eps_n^2, \,  \, V(p_0,p) \le \widetilde\eps_n^2, \, \min_{1\le i \le m_n} p_i \ge \underline{p}_n\right\}\right).
\] 
Furthermore, by Lemma \ref{lem:KLVtoL1} from Appendix~\ref{app:lemmas}, we have 
\[ U_n \ge \Pi_n\left(\left\{p \colon  \sum_{i=1}^{m_n} |p_{0i}-p_i| \le 2 g(\widetilde\eps_n, \underline{p}_n), \, \min_{1\le i \le m_n} p_i \ge \underline{p}_n\right\}\right).\]
The statement of the lemma now follows upon  applying Lemma \ref{lem:dirichlet} from Appendix \ref{app:lemmas} with $\eta=\underline{p}_n$ and $\eps =g(\widetilde\eps_n, \underline{p}_n)$. 
\end{proof}

\subsubsection{Using bounds in Theorem \ref{thm:ggvdv}}
We take
\begin{gather*}
m_n \asymp \log n, \quad \eps_n \asymp \frac{\log^{\gamma} n}{\sqrt{n}},\quad \underline{p}_n \asymp \frac{1}{n^2},\\
\overline{\Lambda}_n \asymp \log^{2\gamma} n, \quad  \underline{\Lambda}_n \asymp \exp(-\operatorname{const} \cdot \log^{2\gamma}n)
\end{gather*}
with appropriately selected proportionality constants, and verify the conditions in Theorem \ref{thm:ggvdv}.

Firstly, condition \eqref{eq:nlarge} is trivially satisfied. Therefore, we can invoke Lemma \ref{lem:entropy} and conclude that the entropy is upper bounded by a multiple of $\log^{2\gamma} n$, since $\gamma > 1.$ Now $\log^{2\gamma} n \lesssim n\eps_n^2$, and this verifies the entropy condition in Theorem \ref{thm:ggvdv}.

Be Lemma \ref{lem:remainingmass}, for a suitable choice of the constant $C$ the remaining prior mass condition is likewise satisfied.

Finally, for the prior mass condition in a small Kullback-Leibler neighbourhood to hold, by Lemma \ref{lem:priormass} we need that the term
\[
\Pi_n\left(  |\lambda_0-\lambda| \le \widetilde\eps_n\right)  \exp\left(-m_n\log(1/(g(\widetilde\eps_n,\underline{p}_n)^2-\underline{p}_n)) - m_n \log(1/\underline{\alpha})\right)
\]
is lower bounded by
 $ \exp( - C n\eps_n^2)$ for some large enough $C>0$. Now, $\Pi_n\left(   |\lambda_0-\lambda| \le \widetilde\eps_n\right)  \asymp \widetilde\eps_n.$ Take  the logarithm on both sides of the above display and note that by our conditions 
 \[
 \log \left( \Pi_n\left(  |\lambda_0-\lambda| \le \widetilde\eps_n\right) \right) \gtrsim \log( \widetilde\eps_n) \gtrsim -n\eps_n^2.
 \]
 Likewise,
 \[
 m_n\log(1/(g(\widetilde\eps_n,\underline{p}_n)^2-\underline{p}_n)) + m_n \log(1/\underline{\alpha}) \lesssim n\eps_n^2,
 \]
so that the prior mass condition holds.
 
Thus we have verified all the conditions of Theorem \ref{thm:ggvdv}. The resulting posterior contraction rate is $\eps_n \asymp {\log^\gamma n}/{\sqrt{n}}.$

\section{Outlook}
\label{sec:outlook}

In this paper we introduced a non-parametric Bayesian approach to estimation of the L\'evy measure $\nu$ of a discretely observed CPP, when the support of $\nu$ is a subset of $\mathbb{N}.$ We constructed an algorithm for sampling from the posterior distribution of $\nu$, and showed that in practice our procedure performs well and measures up to a benchmark frequentist plug-in approach from \cite{buchmann04}. Although computationally more demanding and slower than the latter, our method has an added benefit of providing uncertainty quantification in parameter estimates through the spread of the posterior distribution. On the theoretical side we show that our procedure is consistent, in that  asymptotically, as the sample size $n\rightarrow\infty,$ the posterior concentrates around the `true', data-generating distribution. The corresponding posterior contraction rate is the (nearly) optimal rate $\log^{\gamma}n/\sqrt{n}$ for an arbitrary $\gamma > 1$, if we are to ignore a practically insignificant $\log n$ factor.

Among several generalisations of our results, the one that looks the most promising is extension of our methodology to CPP processes with jump size distributions supported on the set of integers $\mathbb{Z}.$ The corresponding model has garnered substantial interest in financial applications, see \cite{barndorff12}. We leave this extension as a topic of possible future research.

\section*{Acknowledgements}
The authors would like to thank the Associate Editor and the referee for their detailed and constructive comments on the paper.

\bibliographystyle{italic-apa-good}
\bibliography{bibliography}

\appendix
\section{Technical results}
\label{app:lemmas}

\begin{lemma}\label{KLV}
	Let $q$ (resp. $q^{\prime}$) be the law at time $1$ of a compound Poisson process with intensity $\lambda$ (resp. $\lambda^{\prime}$) and jump distribution $p$ (resp. $p^{\prime}$). Suppose that $p$ and $p^{\prime}$ are distributions concentrated on $\N$. Then,
	\begin{align*}
	KL(q,q')&\leq \lambda KL(p,p')+\lambda'-\lambda+\lambda\log\frac{\lambda}{\lambda'},\\
V(q,q')&\leq 2\lambda (V(p,p')+2KL(p,p'))+2KL(p,p')^2\lambda^2 \nonumber \\ 
	&\quad+2\lambda\log\left(\frac{\lambda}{\lambda'}\right)\left(2(\lambda'-\lambda)+(\lambda-1)\log\left(\frac{\lambda}{\lambda'}\right)\right)+2(\lambda'-\lambda)^2,\\
		h(q,q')&\leq \sqrt \lambda h(p,p')+\sqrt{1-e^{-\frac{1}{2}(\sqrt{\lambda}-\sqrt{\lambda'})^2}}\leq \sqrt\lambda h(p,p')+|\sqrt\lambda-\sqrt{\lambda'}|.
	\end{align*}
	In particular, if $\lambda,\lambda^{\prime}\in[\underline{\Lambda}, \overline\Lambda]$ with $0<\underline{\Lambda}\leq \overline\Lambda<\infty$, then there exists a positive constant $\overline{C}$, that depends on $\underline{\Lambda}$, $\overline\Lambda$, such that
	\begin{align*}
	KL(q,q^{\prime})&\leq \overline{C}\left( KL(p,p^{\prime})+|\lambda - \lambda^{\prime}|\right),\\
	V(q,q^{\prime})&\leq \overline{C}\left(V(p,p')+KL(p,p')+(\lambda - \lambda^{\prime})^2\right),\\
	h(q,q^{\prime})& \leq \overline{C} h(p,p^{\prime}) + |\sqrt{\lambda} - \sqrt{\lambda^{\prime}}|.
	\end{align*}
\end{lemma}

\begin{proof}
	If $KL(p,p^{\prime})$ and $V(p,p^{\prime})$ are infinite, then the above inequalities are trivially satisfied. Therefore, we can assume these two divergences are finite. With this in mind, the proof of the lemma is divided into three steps.
	
		\textbf{Step 1:} We begin by proving that for any $n\geq 1,$
	\begin{align*} 
	KL\big(p^{*n},p'^{*n}\big)&\leq n KL(p,p'),\\
        V\big(p^{*n},p'^{*n}\big)&\leq n V(p,p') +4nKL(p,p')+n(n-1)KL(p,p')^2,\\
        h^2\big(p^{*n},p'^{*n}\big)&\leq n h^2(p,p').
	\end{align*}
	The assertions are trivial for $n=1$. Assuming that the first one holds for $n-1$ with $n\geq 2$, we will now show that it holds for $n$ as well. Using the notation $p^{*n}(i)$ for the $i$th element of $p^{*n}$ and similarly in the case of $p^{\prime*n}$, we have
	\begin{align*}
	KL\big(p^{*n},p'^{*n}\big)&=\sum_{i\in\N} p^{*n}(i)\log \bigg(\frac{p^{*n}(i)}{p'^{*n}(i)}\bigg)\nonumber \\
	&=\sum_{i\in\N}\sum_{k\in\N} p^{*(n-1)}(k) p(i-k)\log \bigg(\frac{\sum_{k\in\N} p^{*(n-1)}(k) p(i-k)}{\sum_{k\in\N} p'^{*(n-1)}(k) p'(i-k)}\bigg)\nonumber\\
	&\leq \sum_{i\in\N}\sum_{k\in\N}p^{*(n-1)}(k) p(i-k)\log \bigg(\frac{p^{*(n-1)}(k) p(i-k)}{p'^{*(n-1)}(k) p'(i-k)}\bigg) \\
	&=\sum_{i\in\N}\sum_{k\in\N} p^{*(n-1)}(k) p(i-k)\bigg(\log \bigg(\frac{p^{*(n-1)}(k)}{p'^{*(n-1)}(k)}\bigg)+\log\bigg(\frac{p(i-k)}{p'(i-k)}\bigg)\bigg)\\ \nonumber
	&=KL\big(p^{*(n-1)},p'^{*(n-1)}\big)+KL(p,p'), 
	\end{align*}
	where the inequality follows from the log-sum inequality, and the last equality is obtained by means of Fubini's theorem combined with the facts that
	\[
	\sum_{k\in\N} p^{*(n-1)}(k)=1, \quad \sum_{i\in\N}p(i-k)=1, \quad \forall k\in\N.
	\]
	By induction, we deduce that $KL\big(p^{*n},p'^{*n}\big)\leq n KL(p,p')$. 
	
	The proof of the inequality for $V$ is similar: we assume the inequality is true for $n-1$ with $n\geq 2$, and will show it holds for $n$ as well. Write $V\left(p^{*n},p'^{*n}\right)=R_1+R_2$ for
	\begin{align*}
	R_1&=\sum_{i\in\N} p^{*n}(i)\log^2\bigg(\frac{p^{*n}(i)}{p'^{*n}(i)}\bigg) \ind_{ \left\{ \frac{p^{*n}(i)}{p'^{*n}(i)}\geq 1 \right\} },\\
	R_2&=\sum_{i\in\N} p^{*n}(i)\log^2\bigg(\frac{p^{*n}(i)}{p'^{*n}(i)}\bigg) \ind_{ \left\{ \frac{p^{*n}(i)}{p'^{*n}(i)}<1 \right\} }.
	\end{align*}
	Observe that the function $x \mapsto (x\log^2x)\ind_{ \{ x\geq 1 \} }$ is convex. By Jensen's inequality we have for positive $a_k,b_k$ that
	\[
	\left( \sum_k a_k \right) \log^2\left( \frac{ \sum_k a_k }{ \sum_k b_k } \right) \ind_{ \left\{ \sum_k a_k \geq \sum_k b_k \right\} }
	\leq \sum_k a_k \log^2\left( \frac{a_k}{b_k} \right) \ind_{ \left\{ a_k/b_k \geq 1 \right\} }. 
	\]
	Using this inequality and
	\[
	p^{*n}(i) = \sum_{k\in\mathbb{N}} p^{*(n-1)}(k) p(i-k), \quad p'^{*n}(i) = \sum_{k\in\mathbb{N}} p'^{*(n-1)}(k) p'(i-k),
	\]
	we get that
	\begin{align*}
	R_1
	&\leq \sum_{i\in\N} \sum_{k\in\N} p(i-k)p^{*(n-1)}(k)\log^2\bigg(\frac{p(i-k)p^{*(n-1)}(k)}{p'(i-k)p'^{*(n-1)}(k)}\bigg)\\
	&=\sum_{i\in\N} \sum_{k\in\N} p(i-k)p^{*(n-1)}(k)\bigg(\log\bigg(\frac{p(i-k)}{p'(i-k)}\bigg)+\log\bigg(\frac{p^{*(n-1)}(k)}{p'^{*(n-1)}(k)}\bigg)\bigg)^2\\
	&=V\big(p^{*(n-1)},p'^{*(n-1)}\big)+V(p,p')+2KL(p,p')KL\big(p^{*(n-1)},p'^{*(n-1)}\big)\\
	&\leq V\big(p^{*(n-1)},p'^{*(n-1)}\big)+V(p,p') +2(n-1)KL(p,p')^2\\
	&\leq nV(p,p^{\prime}) + 4(n-1)KL(p,p^{\prime}) + n(n-1)KL(p,p^{\prime})^2,
	\end{align*}
	where in the last inequality we used the induction hypothesis.
	Now recall an elementary inequality
	\[
	e^{-x}x^2 \leq 4 \left( e^{-x/2} - 1 \right)^2
	\]
	valid for $x \geq 0$; see p.~12 in \cite{gugu15}. Applying this inequality to
	\[
	x = - \log\left(\frac{p^{*n}(i)}{p'^{*n}(i)}\right)
	\]
	such that ${p^{*n}(i)}/{p'^{*n}(i)}<1,$ we get
	\[
	\frac{p^{*n}(i)}{ p'^{*n}(i) } \log^2\left(\frac{p^{*n}(i)}{p'^{*n}(i)}\right) \leq 4 \left( \sqrt{ \frac{p^{*n}(i)}{p'^{*n}(i)} } - 1 \right)^2.
	\]
	By multiplying both sides of the above inequality with $p'^{*n}(i),$ summing the result through $i$ and recalling the definition of the Hellinger distance, we get that
	\begin{align*}
	R_2 & \leq 4 \sum_{i\in\mathbb{N}} \left( \sqrt{ p^{*n}(i) } - \sqrt{ p'^{*n}(i) } \right)^2 = 4 h^2\left( p^{*n},p'^{*n} \right) \\
	& \leq 4KL\left(p^{*n},p'^{*n}\right) \leq 4 n KL(p,p').
	\end{align*}
	To conclude the proof of the inequality for $V$, we combine the bounds derived for $R_1$ and $R_2.$ 
	
	As far as the inequality for the Hellinger distance is concerned, we observe that
	\[
	h^2(p^{*n},p'^{*n})=\sum_{i\in\N} p^{*n}(i)g\left(\frac{p'^{*n}(i)}{p^{*n}(i)}\right)
	\]
	for a convex function $g(x)=(1-\sqrt x)^2 \ind_{[0,\infty)}(x).$ Then, by the same reasoning as above, we have
	\begin{align*}
	h^2\big(p^{*n},p'^{*n}\big)&\leq \sum_{i\in\N} \sum_{k\in\N} p^{*(n-1)}(k)p(i-k)g\left(\frac{p'^{*(n-1)}(k)p'(i-k)}{p^{*(n-1)}(k)p(i-k)}\right)\\
	&=\sum_{i\in\N} \sum_{k\in\N}\bigg(\left(\sqrt{p^{*(n-1)}(k)}-\sqrt{p'^{*(n-1)}(k)}\right)\sqrt{p(i-k)}\\
	&\phantom{=\sum_{i,k\in\N}\bigg(\Big(}+\sqrt{p'^{*(n-1)}(k)}\big(\sqrt{p(i-k)}-\sqrt{p'(i-k)}\big)\bigg)^2\\
	&=h^2(p^{*(n-1)},p'^{*(n-1)})+h^2(p,p')\\
	&\quad +2 \sum_{i\in\N} \sum_{k\in\N} \left(\sqrt{p^{*(n-1)}(k)}-\sqrt{p'^{*(n-1)}(k)}\right)\\
	&\quad\phantom{+2\sum_{i,k\in \N} \big(} \times(\sqrt{p(i-k)}-\sqrt{p'(i-k)})\sqrt{p(i-k)p^{*(n-1)}(k)}.
	\end{align*}
	Now note that the last summand satisfies
	\begin{multline*} 
2\left(\sum_{k\in\N} \sqrt{p^{\prime*(n-1)}(k)p^{*(n-1)}(k)}-1\right)\left(1-\sum_{k\in\N}\sqrt{p(k)p'(k)}\right)\\ =-\frac{1}{2}h^2( p^{*(n-1)} , p^{\prime*(n-1)}(k) ) h^2(p,p^{\prime}) \leq 0.
	\end{multline*}
	We therefore conclude that
	 $$h^2\big(p^{*n},p'^{*n}\big)\leq h^2(p^{*(n-1)},p'^{*(n-1)})+h^2(p,p'),$$
	 which leads to the desired inequality, by an induction argument. 
	
	\textbf{Step 2: } Now we prove the inequalities
	\begin{align*}
	KL(q,q')&\leq \sum_{n=0}^\infty\pp(N=n)KL\left(p^{*n},p'^{*n}\right) +KL(N,N'),\\
	V(q,q')&\leq2\sum_{n=0}^\infty\pp(N=n)\left(V\left(p^{*n},p'^{*n}\right)+2KL\left(p^{*n},p'^{*n}\right)\right) +2V(N,N'),\\
	h(q,q')&\leq \sqrt{\sum_{n=0}^\infty \pp(N=n)h^2\left(p^{*n},p'^{*n}\right)}+h(N,N').
	\end{align*}
	Here $N$ and $N'$ are Poisson random variables with means $\lambda$ and $\lambda'$, respectively, and with a slight abuse of notation, $KL(N,N')$, $V(N,N')$ and $h(N,N')$ are the $KL$ and $V$ divergences and the Hellinger distance between the corresponding laws.
	
	Note that
	\[
	q(i) = \sum_{n=0}^{\infty} p^{*n}(i)P(N=n), \quad q^{\prime}(i) = \sum_{n=0}^{\infty} p^{\prime*n}(i)P(N'=n).
	\]
	Using this and the log-sum inequality,
	\begin{align*}
	KL(q,q')&=\sum_{i\in\N} q(i)\log\left(\frac{q(i)}{q'(i)}\right)
	\leq\sum_{i\in\N} \sum_{n\in\N} p^{*n}(i) P (N=n)\log\left(\frac{ p^{*n}(i) P(N=n)}{ p'^{*n}(i) P(N'=n)}\right)\nonumber\\
	&=\sum_{i\in\N} \sum_{n\in\N}  p^{*n}(i)P(N=n)\left(\log\left(\frac{ p^{*n}(i)}{ p'^{*n}(i)}\right)+\log\left(\frac{P(N=n)}{P(N'=n)}\right)\right)\\ 
	&=\sum_{n=0}^\infty P(N=n)KL\left( p^{*n}, p'^{*n}\right) +KL(N,N').
	\end{align*}
	
	For the divergence $V,$ write $V(q,q')=B_1+B_2$ for
	\begin{align*}
	B_1&=\sum_{i\in\N} q(i) \log^2\left(\frac{q(i)}{q'(i)}\right)\ind_{\left\{\frac{ q(i)}{ q'(i)}\geq 1\right\}} \\
	      &\leq \sum_{i\in\N} \sum_{n\in\N} p^{*n}(i)\pp(N=n)\log^2\left(\frac{p^{*n}(i)\pp(N=n)}{p'^{*n}(i)\pp(N'=n)}\right)\\
	      &\leq 2 \sum_{i\in\N} \sum_{n\in\N} p^{*n}(i)\pp(N=n)\left(\log^2\left(\frac{p^{*n}(i)}{p'^{*n}(i)}\right)+\log^2\left(\frac{\pp(N=n)}{\pp(N'=n)}\right)\right)\\
	      &=2\sum_{n=0}^\infty V(p^{*n},p'^{*n})\pp(N=n)+2V(N,N'),\\
        B_2&=\sum_{i\in\N} q(i) \log^2\left(\frac{q(i)}{q'(i)}\right)\ind_{\left\{\frac{ q(i)}{ q'(i)}<1\right\}}.
	\end{align*}
	To control $B_2$, we use the same arguments as in the proof of inequalities (12) and (15) in \cite{gugu15}, getting 
$$B_2\leq 4KL(q,q')\leq 4 \sum_{n=0}^\infty KL(p^{*n},p'^{*n})\pp(N=n).$$
This gives the required inequality for the $V$ divergence.

	Finally, we prove the inequality for the Hellinger distance. Denoting the law of $\sum_{j=1}^N Y_j$ by $\tilde q$, it holds by the triangle inequality that
	$h(q,q')\leq h(q,\tilde q)+h(\tilde q,q').$ Since $g(x)=(1-\sqrt{x})^2\ind_{[0,\infty)}(x)$ is a convex function,
	\begin{align*}
	h^2(q,\tilde q)&\leq \sum_{i\in\N} \sum_{n\in\N} p^{*n}(i)\pp(N=n) g\left(\frac{p'^{*n}(i)}{p^{*n}(i)}\right)=\sum_{n\in\N}\pp(N=n)h^2( p^{*n}, p'^{*n}).
	\end{align*}
	It remains to prove that $h(\tilde q,q')\leq h(N,N').$ This again follows by convexity of $g$, since
	\begin{align*}
	h^2(\tilde q,q')&=\sum_{i\in\N} \sum_{n\in\N} p'^{*n}(i)\pp(N=n)g\left(\frac{\sum_{n=0}^\infty p'^{*n}(i)\pp(N'=n)}{\sum_{n=0}^\infty p'^{*n}(i)\pp(N=n)}\right)\\
	&\leq \sum_{i\in\N} \sum_{n\in\N} p'^{*n}(i)\pp(N=n)g\left(\frac{\pp(N'=n)}{\pp(N=n)}\right)=h^2(N,N').
	\end{align*}
	
	\textbf{Step 3:} From Steps 1 and 2 we derive that
	\begin{align*}
	KL(q,q')&\leq KL(p,p') \ee[N]+KL(N,N'),\\
	V(q,q')&\leq 2 \ee[N](V(p,p')+2KL(p,p'))+2V(N,N')\\
	&\quad +2(\ee[N^2]-\ee[N])KL(p,p')^2,\\
	h(q,q')&\leq h(p,p')\sqrt{\ee[N]}+h(N,N').
	\end{align*}
	
	Now the proof of the lemma follows from these three inequalities upon noticing that $KL(p,p')^2\leq V(p,p')$, and recalling that
	\begin{align*}
	\ee[N]&=\lambda, \\
	\ee[N^2]&=\lambda+\lambda^2,\\
	KL(N,N') & =\lambda'-\lambda+\lambda\log\frac{\lambda}{\lambda'}\lesssim (\lambda-\lambda')^2,\\
	h^2(N,N')&=1-e^{-\frac{1}{2}(\sqrt{\lambda}-\sqrt{\lambda'})^2}\leq |\sqrt\lambda-\sqrt{\lambda'}|,\\
	V(N,N')&=\lambda\log\left(\frac{\lambda}{\lambda'}\right)\left(2(\lambda'-\lambda)+(\lambda-1)\log\left(\frac{\lambda}{\lambda'}\right)\right)+(\lambda'-\lambda)^2\lesssim (\lambda-\lambda')^2.
	\end{align*}
\end{proof}

\begin{lemma}\label{lem:KLVtoL1}
	Let $\eps>0$. 
	Suppose $p=(p_{1},\ldots, p_{m})$ and  $p^{\prime}=(p^{\prime}_1,\ldots, p^{\prime}_m)$ are points in the $m$-dimensional unit simplex, and let $\min_{1\le i \le m} p_i \ge c$ for some $c \in (0,1)$. Then there exists a universal constant $C>0$, such that the inequality
	\[ \sum_{i=1}^m | p^{\prime}_i-p_i| \le C \frac{\eps^2}{[\log (e/c)]^2} \]
	implies that $KL(p^{\prime}, p) \le \eps^2$ and $V(p^{\prime},p) \le \eps^2$ hold. 
\end{lemma}
\begin{proof}
	Lemma 8 in \cite{ghosal07} assures that there exists a  constant $\overline{C}$ (not depending on either $p$ or $p^{\prime}$), such that 
	\[  KL(p^{\prime}, p) \le \overline{C} h^2(p^{\prime},p) \left[ 1 + \log\left( \left\| \frac{p^{\prime}}{p} \right\|_{\infty}\right) \right] \]
	and
	\[  V(p^{\prime}, p) \le \overline{C} h^2(p^{\prime},p)  \left[ 1 + \log\left( \left\| \frac{p^{\prime}}{p} \right\|_{\infty}\right) \right]^2. \]
	From Section 3.3 in \cite{pollard_2001} we have $h^2(p^{\prime},p)\le  \sum_{i=1}^m |p^{\prime}_i-p_i|$. 
	Furthermore, since $0<c<1$ and $\min_{1\le i \le m} p_i \ge c$,
	\[ 1 \le 1 + \log\left( \left\| \frac{p^{\prime}}{p} \right\|_{\infty}\right)  \le 1  +\log (1/c)=\log (e/c). \]
	Therefore,
	\[ \max(KL(p^{\prime},p), V(p^{\prime},p)) \le \overline{C}   \left[\log(e/c)\right]^2 \sum_{i=1}^m |p^{\prime}_i-p_i| , \]
	from which the assertion of the lemma follows trivially.
\end{proof}

\begin{lemma}\label{lem:dirichlet}
	Let $m\ge 2$ be an integer.	Suppose $(p_1,\ldots,p_m) \sim \operatorname{Dir}(\alpha_1,\ldots, \alpha_m)$. Let $p_0=(p_{01},\ldots, p_{0m})$ be an arbitrary point in the $m$-dimensional unit simplex. Assume there exists  $\underline\alpha$, such that $0<\underline{\alpha} \le \alpha_i \le 1$ for all $1\le i \le m$. Let $\eps>0$, and let $\eta$ be such that $\eta<\eps^2$. Then if $m\eps<1$,
	\begin{equation}\label{eq:dirineq}
	\begin{split}
	&	\Pi_n\left(\sum_{i=1}^m |p_i - p_{0i}| \le 2 \eps,\,   \min_{1\le i \le m} p_{i} \ge \eta \right)  \\ &\qquad \qquad  \ge  
	\Gamma\left(\sum_{i=1}^m \alpha_i\right)  \exp\left(-m\log(1/(\eps^2-\eta)) - m \log(1/\underline{\alpha})\right).
	\end{split}	
	\end{equation}
\end{lemma}
\begin{proof}
	By arguments analogous to those in the proofs of Lemma 6.1 in \cite{ghosal00} and Lemma 10 in \cite{ghosal07}, we obtain that the left-hand side of \eqref{eq:dirineq} can be lower bounded by 
	\[ \frac{\Gamma(\sum_{i=1}^m \alpha_i)}{\prod_{i=1}^m \Gamma(\alpha_i)} \prod_{i=1}^{m-1} \int_{\max(p_{0i}-\eps^2, \eta^2)}^{\min(p_{0i}+\eps^2,1)} x_i^{\alpha_i-1} \dd x_i. \]
	The length of the integration interval in each of the integrals in the above product is lower bounded by $\eps^2-\eta$. Using that $\underline{\alpha} \le \alpha_i\le 1$, we deduce that the preceding display is lower bounded by 
	\[ \Gamma\left(\sum_{i=1}^m \alpha_i\right) \underline{\alpha}^m \exp\left(-(m-1)\log\left(\frac{1}{\eps^2-\eta}\right)\right). \]
	This entails the result. 
\end{proof}

\begin{figure}
	\centering
	\captionsetup{width=0.8\textwidth, font=small}
	\includegraphics[trim=0.6cm 0cm 0cm 0cm, clip,width=0.8\textwidth]{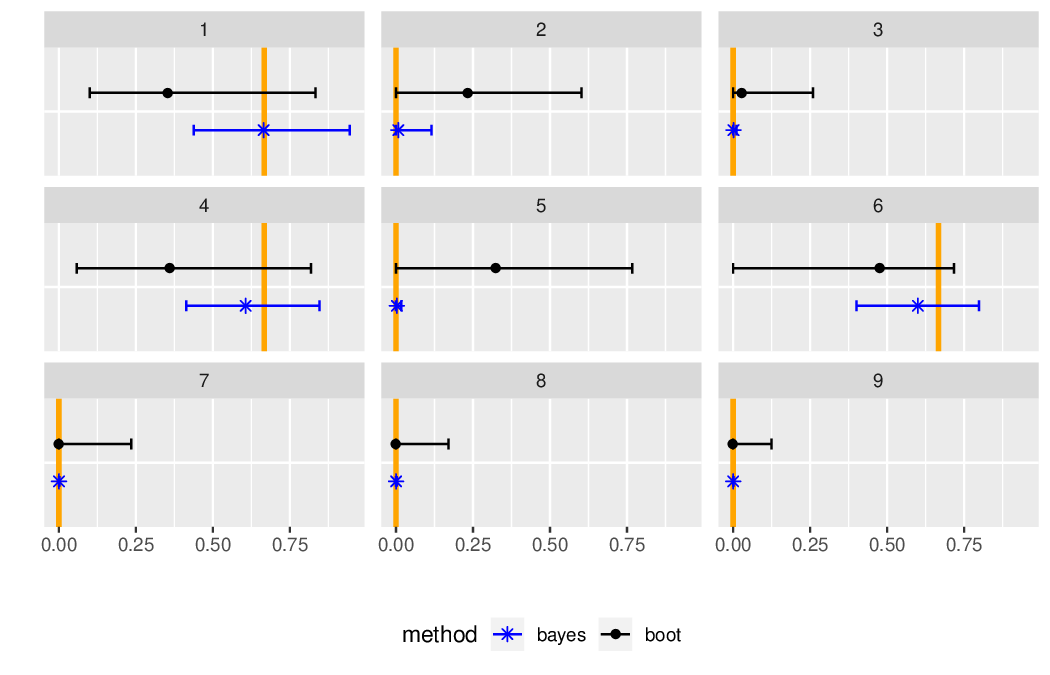}
	\caption{$95\%$ bootstrap confidence intervals and Bayesian credible intervals for the simulation example from Section \ref{subsec:uniform} under setting (a). The displayed results are for parameters $\nu_1$ through $\nu_9$, with panels labelled sequentially from 1 to 9. The true parameter values are visualised with orange vertical lines. The colouring scheme is the same as in Figure \ref{fig:example1}. Note that some of the narrow Bayesian credible intervals are overshadowed by the symbol (star) used to visualise the posterior mean.
		\label{fig:example1-ci}}
\end{figure}

\section{Bootstrap confidence intervals}
\label{app:bg:boot}

Here we report a small comparison between the Bayesian credible intervals and the bootstrap confidence intervals for the Buchmann-Gr\"{u}bel estimator. The setup and the simulated dataset that we used are the same as in Subsection~\ref{subsec:uniform}. The bootstrap confidence intervals were computed as follows: $B=9999$ bootstrap samples were generated from the compound Poisson model under the Buchmann-Gr\"{u}bel estimates computed from the observed data. These bootstrap samples were then fed back to the Buchmann-Gr\"{u}bel procedure to yield $B$ bootstrap estimates of the L\'evy measure $\nu$. Finally, for each $k$, the $\alpha/2$-th and $(1-\alpha/2)$-th sample quantiles were obtained to yield $1-\alpha$ bootstrap confidence intervals for $\nu_k.$ The results with $\alpha=0.05$ are displayed in Figure \ref{fig:example1-ci}. One observes that while both methods result in a good coverage for this specific dataset, the bootstrap appears to be noticeably more conservative than the Bayesian approach, as evidenced by the width of the intervals.

\end{document}